%% file: main.tex
\documentclass[11pt]{article}

\input{preamble}
\usepackage{amsmath,amssymb,xfrac}
    \usepackage{enumitem} \setlist{nosep}
\usepackage[pagebackref,breaklinks,unicode]{hyperref} 
    \renewcommand*{\backrefalt}[4]{\ifcase #1 (Not cited).\or (Cited p.~#2).\else (Cited pp.~#2).\fi} 
\usepackage{graphicx} 
\usepackage{mathrsfs}
\usepackage[T1]{fontenc}
\usepackage[margin=3cm]{geometry}
\usepackage{natbib}

\newcommand*{\ext}{\mathrm{ext}}
\newcommand*{\N}{\mathbb{N}}

\makeatletter
\newtheorem*{rep@theorem}{\rep@title}
\newcommand{\newreptheorem}[2]{%
\newenvironment{rep#1}[1]{%
 \def\rep@title{#2 \ref{##1}}%
 \begin{rep@theorem}}%
 {\end{rep@theorem}}}
\makeatother

\newreptheorem{theorem}{Theorem}
\newreptheorem{example}{Example}
\newreptheorem{fact}{Fact}

\newcommand*{\ssm}{\smallsetminus}

\definecolor{harrycomment}{rgb}{0.6,0,0.4}

\definecolor{shakedcomment}{rgb}{0, 0, 255}

\definecolor{oussamacomment}{rgb}{0,0.5,0}

\definecolor{dawidcomment}{rgb}{1,0.5,0}

\title{Quasiisometric embeddings between right-angled Artin groups: flexibility}
\hypersetup{pdftitle = Quasiisometric embeddings of RAAGs: flexibility}
\date{}
\author{Shaked Bader, Oussama Bensaid, and Harry Petyt}

\newcommand{\Addresses}{{\bigskip\footnotesize\par
\textsc{Mathematical Institute, University of Oxford, UK}\par\nopagebreak\textit{E-mail address}: 
\texttt{shaked.bader@sjc.ox.ac.uk}
\par\medskip\par
\textsc{Institut de Recherche en Mathématique et en Physique, Université Catholique de Louvain, Belgium} \par\nopagebreak\textit{E-mail address}: 
\texttt{oussama.bensaid@uclouvain.be}
\par\medskip\par
\textsc{Mathematics Institute, University of Warwick, UK}\par\nopagebreak\textit{E-mail address}: 
\texttt{harrypetyt@gmail.com}
}}

\begin{document}

\maketitle

\begin{abstract}
We give a complete characterisation of when the right-angled Artin group on one cycle graph can be quasiisometrically embedded in the right-angled Artin group on another cycle graph. In particular, we find infinitely many instances of quasiisometric embeddings where there is no subgroup relation. This contrasts with the fact that such groups are quasiisometrically rigid.

More generally, we construct quasiisometric embeddings between graph products of finite or cyclic groups whose underlying graphs are cycles. As a special case, we obtain exotic quasiisometric embeddings of the hyperbolic plane in all right-angled Artin groups whose defining graph contains an induced cycle of length greater than four.
\end{abstract}

{\hypersetup{hidelinks}\setcounter{tocdepth}{1}\tableofcontents\setcounter{tocdepth}{2}}

\section{Introduction}

A \emph{right-angled Artin group} is a group whose only relations are commutators of some generators. At one extreme of this we have free groups and at the other we have free-abelian groups, and one can view the class of right-angled Artin groups as an interpolation between these two extremes. 

Accordingly, right-angled Artin groups have many nice algebraic and combinatorial properties reminiscent of free and free-abelian groups. For example, they admit simple normal forms for elements \cite{hermillermeier:algorithms}, their conjugacy classes are well understood \cite{servatius:automorphisms}, and they are linear \cite{davisjanuszkiewicz:right,hsuwise:onlinear}. See \cite{charney:introduction} for a survey.

In many ways, though, thinking of right-angled Artin groups as an interpolation between free and free-abelian groups undersells the richness and complexity of the class. This is perhaps most glaring when one considers their subgroup structure. As an early example, Bestvina--Brady famously used subgroups of right-angled Artin groups to show the inconsistency of the Eilenberg--Ganea and Whitehead conjectures \cite{bestvinabrady:morse}. Furthermore, the \emph{compact special} groups of Haglund--Wise can be characterised as the \emph{convex cocompact} subgroups of right-angled Artin groups \cite{haglundwise:special}. Specialness has been an extremely influential notion, for instance playing a key role in the resolution of the virtual Haken conjecture \cite{agol:virtual,wise:structure}. Aside from subgroups, Croke--Kleiner gave proper, cocompact actions of a four-generator right-angled Artin group on two CAT(0) spaces with non-homeomorphic visual boundaries \cite{crokekleiner:spaces}, and more recently Fioravanti has shown that understanding the growth of automorphisms of right-angled Artin groups requires understanding automorphisms of all compact special groups \cite{fioravanti:growth}.

One can ask when two right-angled Artin groups are isomorphic, and it turns out that this has a simple answer. A right-angled Artin group $A$ can be encoded by a simplicial graph $\Gamma$ that has a vertex for each (standard) generator and an edge whenever two generators commute. Clearly, if two right-angled Artin groups have the same associated graph then they are isomorphic, and Droms proved the converse \cite{droms:isomorphisms}. Thus there is a one-to-one correspondence between right-angled Artin groups and simplicial graphs, so we can write $A=A_\Gamma$ without ambiguity. 

A more subtle question is that of when one right-angled Artin group is a subgroup of another. The most obvious subgroups of $A_\Gamma$ are the \emph{parabolic} subgroups $A_\Lambda$, where $\Lambda$ is an induced subgraph of $\Gamma$. This can be pushed further: if one \emph{doubles} $\Gamma$ (see \cref{def:double}), then one obtains an index-two subgroup of $A_\Gamma$, so one can obtain right-angled Artin subgroups by repeatedly doubling and taking an induced subgraph. This naturally leads to the \emph{extension graph} $\Gamma^\ext$, introduced by Kim--Koberda \cite{kimkoberda:embedability}. If $\Lambda$ is an induced subgraph of $\Gamma^\ext$, then $A_\Lambda$ is a subgroup of $A_\Gamma$ \cite[Thm~1.3]{kimkoberda:embedability} --- we give an account of a refinement of this in \cref{sec:induced_implies_subgroup}. 

Kim--Koberda also studied the converse, and showed that the existence of a monomorphism $A_\Lambda\to A_\Gamma$ puts restrictions on $\Lambda^\ext$ and $\Gamma^\ext$. Their results are strongest in the case where $\Gamma$ has no triangles, in which case they show that it is necessary for $\Lambda$ to be an induced subgraph of $\Gamma^\ext$ \cite[Thm~1.11]{kimkoberda:embedability}. This completely classifies which right-angled Artin groups are subgroups of $A_\Gamma$ when $\Gamma$ is triangle-free, and their theorem can be viewed as a natural extension of Droms' theorem. By studying the extension graphs of cycles, Kim--Koberda deduce the following.

\begin{theorem}[{\cite[Thm~1.12]{kimkoberda:embedability}}] \label{thm:kimkoberda_subgps}
Let $m,n\ge4$. There is a monomorphism $A_{C_m}\to A_{C_n}$ if and only if $m=n+p(n-4)$ for some $p\ge0$.
\end{theorem}

From the perspective of large-scale geometry, the two ends of our spectrum are rather different: free-abelian groups are absolutely quasiisometrically rigid, in the sense that every group quasiisometric to $\Z^n$ is commensurable with $\Z^n$, but any two nonabelian free groups are quasiisometric. This suggests that both quasiisometric rigidity and flexibility should be displayed by right-angled Artin groups.

On the flexibility side, Behrstock--Neumann showed that if $\Gamma$ and $\Lambda$ are finite trees of diameter at least three, then $A_\Gamma$ and $A_\Lambda$ are quasiisometric \cite{behrstockneumann:quasiisometric}. More generally, Behrstock--Januszkiewicz--Neumann classified which right-angled Artin groups defined by ``trees of triangles'' are quasiisometric \cite{behrstockjanuszkiewiczneumann:quasiisometric}.

The first examples on the rigidity side were the \emph{atomic} right-angled Artin groups, due to Bestvina--Kleiner--Sageev \cite{bestvinakleinersageev:asymptotic}. This was vastly generalised by Huang, who showed that if $A_\Gamma$ and $A_\Lambda$ both have finite outer automorphism group, then $A_\Gamma$ and $A_\Lambda$ are quasiisometric if and only if they are isomorphic \cite{huang:quasiisometric:1}. He moreover handled various cases with infinite outer automorphism group \cite{huang:quasiisometric:2}, and gave a large class of right-angled Artin groups $A_\Gamma$ that are absolutely quasiisometrically rigid \cite{huang:commensurability}.

The most general known result on the quasiisometry classification of right-angled Artin groups is due to Margolis, and is stated in terms of \emph{JSJ trees of cylinders} \cite{margolis:quasiisometry}.

Extension graphs also play a role in Huang's work \cite{huang:quasiisometric:1}, which involves passing from a quasiisometry $A_\Gamma\to A_\Lambda$ to an isomorphism $\Gamma^\ext\to\Lambda^\ext$. This mirrors the situation in mapping class groups, where quasiisometric rigidity is established by inducing an automorphism of the curve graph \cite{behrstockkleinerminskymosher:geometry} and applying Ivanov's theorem \cite{ivanov:automorphisms,korkmaz:automorphisms,luo:automorphisms}. Extension graphs can often be thought of as acting as curve graphs for right-angled Artin groups \cite{kimkoberda:geometry}.

In this article, we are interested in quasiisometric \emph{embeddings} between right-angled Artin groups. Passing from quasiisometries to quasiisometric embeddings is the geometric analogue of passing from isomorphisms to monomorphisms, which was already discussed. In the setting of mapping class groups, Bowditch proved a strong rigidity result for quasiisometric embeddings: any quasiisometric embedding between mapping class groups of the same complexity at least four is actually a self-quasiisometry at bounded distance from a left-multiplication map \cite{bowditch:large:mapping}. There are also very strong results for homomorphisms when the genera do not differ too much \cite{aramayonasouto:homomorphisms,depool:homomorphisms} (which is a necessary condition \cite{aramayonaleiningersouto:injections}).

For the absolutely quasiisometrically rigid right-angled Artin groups considered in \cite{huang:commensurability}, all self-quasiisometries were completely classified in \cite[Thm~7.24]{baderbensaidpetyt:quasiisometric:rigidity}. Comparing with the mapping class group case, one may hope for a correspondence between quasiisometric embeddings and monomorphisms.

Our main result is the following, which shows that these hopes are misplaced, even in the very simplest cases.

\begin{mthm} \label{mthm:cycles_qie}
Let $m,n\ge4$. There is a quasiisometric embedding $A_{C_m}\to A_{C_n}$ if and only if $m=n$ or there exist $p\ge1$, $q\ge0$ such that $m=n+p(n-4)+q(n-2)$.
\end{mthm}

By \cite{huang:quasiisometric:1}, the groups $\{A_{C_m}\}_{m\in \mathbb{N}}$ are pairwise non-quasiisometric, so for $m\ne n$ the quasiisometric embeddings in \cref{mthm:cycles_qie} are not close to quasiisometries. The theorem also shows that the spectrum of quasiisometric embeddings is richer than the spectrum of subgroups: comparing with \cref{thm:kimkoberda_subgps} gives the following.

\begin{mcor}
For each $n>6$, there are infinitely many values of $m$ for which there is a quasiisometric embedding $A_{C_m}\to A_{C_n}$ but $A_{C_m}$ is not a subgroup of $A_{C_n}$.
\end{mcor}

\begin{proof}
If $2q\not\equiv0$ modulo $n-4$, then $m=n+p(n-4)+q(n-2)$ cannot be expressed as $n+p'(n-4)$ with $p'$ an integer. By \cref{mthm:cycles_qie}, there is a quasiisometric embedding $A_{C_m}\to A_{C_n}$, but \cref{thm:kimkoberda_subgps} shows that $A_{C_m}$ is not a subgroup of $A_{C_n}$. Since $n>6$, there are infinitely many such integers $q$. 
\end{proof}

In the case where $m=n+p(n-4)$, the subgroup embeddings in \cref{thm:kimkoberda_subgps} can actually be taken to be quasiisometric embeddings \cite[Cor.~1.15]{kimkoberda:embedability} (see also \cref{prop:induced_to_homo}). We show in \cref{prop:no_homo_q0} that, even in this case, there are quasiisometric embeddings that are not at finite distance from any composition of self--quasiisometric-embeddings and a homomorphism (that is, not \emph{quasiconjugate} to a homomorphism).

As mentioned, the transition from quasiisometries to quasiisometric embeddings is analogous to the transition from isomorphisms to subgroups. Since the latter situation corresponded to passing from isomorphisms of defining graphs to a more complex situation involving extension graphs, it is reasonable to wonder whether, at least in quasiisometrically rigid cases, one can relate quasiisometric embeddings to maps between extension graphs. 

We established one direction of this for many right-angled Artin groups in \cite[Cor.~6.7]{baderbensaidpetyt:quasiisometric:rigidity}: in many situations, a quasiisometric embedding $A_\Gamma\to A_\Lambda$ induces a graph embedding $\Gamma^\ext\to\Lambda^\ext$, but the image need not be induced. The existence part of \cref{mthm:cycles_qie} goes in the converse direction, by first building a nice embedding of extension graphs. This raises the following natural question.

\begin{question*}
Are there natural conditions on an embedding $f\colon\Gamma^\ext\to\Lambda^\ext$ of extension graphs under which one can construct a quasiisometric embedding $F\colon A_\Gamma\to A_\Lambda$?
\end{question*}

It would seem that any reasonable construction of $F$ from the data of $f$ would require it to then induce the map $f$, which already puts restrictions on what $F$ can look like: see \cite{baderbensaidpetyt:quasiisometric:rigidity}. A more ambitious question is the following.

\begin{question*}
If there exists an embedding $\Gamma^\ext\to\Lambda^\ext$, does that mean there exists a quasiisometric embedding $A_\Gamma\to A_\Lambda$?
\end{question*}

We point out an important subtlety in this question: it is \emph{not} true that the existence of an embedding $\Gamma\to\Lambda^\ext$ implies that $A_\Gamma$ can be quasiisometrically embedded in $A_\Lambda$. Indeed, it is not difficult to see that $C_{2n-2}$ is a (non-induced) subgraph of $C_n^\ext$, but \cref{mthm:cycles_qie} shows that $A_{C_{2n-2}}$ cannot be quasiisometrically embedded in $A_{C_n}$ for any $n>6$. By comparison, note that if $\Gamma$ embeds as an induced subgraph of $\Lambda^\ext$, then $\Gamma^\ext$ does as well.

The construction of quasiisometric embeddings that we employ in order to prove \cref{mthm:cycles_qie} is robust enough that we can also use it to build quasiisometric embeddings between graph products where the underlying graphs are cycles and the vertex groups are finite or cyclic. The following is a simplified version of \cref{cor:graph_prod_finite}, which is itself a straightforward consequence of \cref{thm:cyclic_product_cyclic}. 

\begin{mthm} \label{mthm:graph_prod}
Suppose that $m,n>6$ satisfy $m=n+p(n-4)+q(n-2)$ for some integers $p\ge1$, $q\ge0$. Let $G$ and $H$ be graph products of finite or cyclic groups with underlying graphs $C_m$ and $C_n$, respectively. If every vertex group of $G$ has smaller cardinality than every vertex group of $H$, then there is a quasiisometric embedding $G\to H$.
\end{mthm}

Given an integer $N\ge2$ and a simplicial graph $\Gamma$, let $A_\Gamma(N)$ denote the graph product with underlying graph $\Gamma$ and with all vertex groups being $\sfrac\Z{N\Z}$. If $N=2$, then $A_\Gamma(2)$ is also known as the \emph{right-angled Coxeter group} $W_\Gamma$. If $n>4$, then the hyperbolic plane can be tiled by regular right-angled $n$-gons, and a regular such tiling can be generated by the reflection group $A_{C_n}(2)$, which is thus a virtual surface group. Hence \cref{mthm:graph_prod} yields exotic quasiisometric embeddings of $\mathbb H^2$ as a special case.

\begin{mcor} \label{mcor:exotic_H2}
If $\Gamma$ contains an induced cycle of length greater than four, then there are quasiisometric embeddings $f\colon\mathbb H^2\to A_\Gamma$ such that $f(\mathbb H^2)$ is not contained in any finite neighbourhood of any finite union of surface subgroups of $A_\Gamma$.
\end{mcor}

More generally, given $n>4$ and an $n$--tuple $\bar N=(N_1,\dots,N_n)$ of integers with $N_i\ge2$ for all $i$, let $A_{C_n}(\bar N)$ denote the graph product with underlying graph $C_n$ and the vertex groups being $\sfrac{\Z}{N_i\Z}$ in cyclic order. Similarly to the above tiling of $\mathbb H^2$, one can generate a thick hyperbolic building, the \emph{Bourdon building} $I_{n,\bar N}$ \cite{bourdon:immeubles,bourdon:surimmeubles}, whose chambers are right-angled hyperbolic $n$--gons and whose edges of type $i$ are contained in exactly $N_i$ chambers. The group $A_{C_n}(\bar N)$ is a uniform lattice in $I_{n,\bar N}$.

Just like how right-angled Artin groups defined on cycles are quasiisometrically rigid \cite{huang:quasiisometric:1,huang:commensurability}, the buildings $I_{n,\bar N}$ are quasiisometrically rigid, in the sense that every self-quasiisometry is at finite distance from a self-isometry \cite{bourdonpajot:rigidity,xie:quasiisometric}. Moreover, any two uniform lattices in $I_{n,\bar N}$ are commensurable \cite{haglund:commensurability,shepherd:commensurability}. In contrast to this rigidity, \cref{mthm:graph_prod} shows that there are interesting quasiisometric embeddings between Bourdon buildings.

We note that \cite{graeberkarrerlazarovichstark:surprising} produces circles in the \emph{Morse boundary} of a right-angled Coxeter group $W_\Gamma$ that cannot be seen from the defining graph $\Gamma$. However, those circles arise from parabolic subgroups of the right-angled Coxeter group on the double of $\Gamma$ that are cocompact Fuchsian, unlike the hyperbolic planes in \cref{mcor:exotic_H2}.

\subsection*{Outline of the article}

\cref{sec:prelim} contains background on right-angled Artin groups and extension graphs, as well as the result \cref{thm:qie_induces} from \cite{baderbensaidpetyt:quasiisometric:rigidity}, which lets us pass from a quasiisometric embedding $A_\Gamma\to A_\Lambda$ to an embedding of extension graphs.

In \cref{sec:induced_implies_subgroup}, we give a simple argument for the fact, originally established in \cite{kimkoberda:embedability}, that if $\Gamma$ is an induced subgraph of $\Lambda^\ext$, then $A_\Gamma$ can be found as an undistorted subgroup of $A_\Lambda$. We prove \cref{prop:induced_to_homo}, which is a refinement of this. Having done so, we discuss how the construction in \cref{prop:induced_to_homo} relates to extension graphs. This discussion follows the same overall structure as the arguments in \cref{sec:building_qie}, and is intended to provide a warm-up to that section.

\cref{sec:building_qie} is the technical heart of the paper. In it, we prove the existence part of \cref{mthm:cycles_qie}. The strategy is as follows. 

Given $m$ and $n$ as in \cref{mthm:cycles_qie}, we first identify a base copy of $C_m$ inside $C_n^\ext$, which we refer to as a \emph{block}. This basic block is not an induced subgraph if $q>0$. We then ``develop'' the basic block to obtain a map $f\colon C_m^\ext\to C_n^\ext$, where copies of $C_m$ are sent to blocks inside $C_n^\ext$. Because blocks are not induced subgraphs, this developing procedure requires two operations, which we call \emph{doubling} and \emph{gliding}. The map $f$ is defined in Item~\ref{sh:f_construction}.

Having constructed $f$, in Item~\ref{sh:roots} we ``lift'' it to a map of \emph{syllable-reduced} words, which we then show induces a map $F\colon A_{C_m}\to A_{C_n}$, which ultimately will be the desired quasiisometric embedding. Showing that $F$ is Lipschitz is straightforward, but showing that it is colipschitz is more involved. Roughly, given $g\in A_{C_m}$, we find a particularly nice syllable-reduced word that represents $g$ and interacts well with the construction of $f$. We then analyse how the image of this word can fail to be syllable-reduced, and this lets us understand the $F$-images of prefixes of $g$. The proof that $F$ is a quasiisometric embedding is concluded in \cref{prop:building_cycles_qie}.

\cref{sec:cyclic_subgraphs_of_ext_graphs} handles the other direction of \cref{mthm:cycles_qie}. In view of \cref{thm:qie_induces}, this amounts to understanding exactly when the extension graph of one cycle graph is a subgraph of the extension graph of another. This is achieved by direct combinatorial considerations in \cref{prop:cycles_equation}.

The short \cref{sec:qie_between_cycles} combines the results of the previous sections to prove \cref{mthm:cycles_qie}.

\cref{lem:homo} shows that, in the case $q=0$ of \cref{mthm:cycles_qie}, the quasiisometric embedding constructed in \cref{sec:building_qie} is at finite distance from a homomorphism. In \cref{sec:exotic_when_q=0}, we give a modified version of the construction that yields a quasiisometric embedding that is not at finite distance from any homomorphism.

Finally, in \cref{sec:cyclic_graph_products} we consider more general graph products. We prove \cref{thm:cyclic_product_cyclic} and \cref{cor:graph_prod_finite}, which imply \cref{mthm:graph_prod}, and then deduce \cref{mcor:exotic_H2}.

\subsection*{Acknowledgements}
OB acknowledges support from the FWO and F.R.S.-FNRS under the Excellence of Science (EOS) programme (project ID 40007542). SB and HP thank the Isaac Newton Institute for their hospitality during the programme \emph{Operators, Graphs, Groups}, where some of the work on this paper took place (EPSRC grant EP/Z000580/1).

We are pleased to thank Dawid Kielak for reading a draft of this paper, as well as Matt Cordes and Anne Thomas for useful comments. 

Claude was used to check for typos in the first version of this paper.

\section{Preliminaries} \label{sec:prelim}

For a (simplicial) graph $\Gamma$, the corresponding \emph{right-angled Artin group} (\emph{RAAG})  $A_\Gamma$ is the group with a generator $s_v$ for each vertex $v\in\Gamma$ and a relation $[s_v,s_w]$ whenever $vw$ is an edge of $\Gamma$. We call $S=\{s_v\,:\,v\in\Gamma\}$ the \emph{standard} generating set. This is the only generating set we shall consider in this paper.

Let $|g|$ denote the word length of $g\in A_\Gamma$ with respect to the generating set $S$. A \emph{minimal representative} of $g$ is a word in $S$ that represents $g$ and has length equal to $|g|$.

\begin{proposition}[\cite{hermillermeier:algorithms}] \label{prop:shuffling}
The minimal representatives of $g\in A_\Gamma$ can be obtained from an arbitrary representative by repeatedly shuffling commuting pairs of elements and cancelling inverse pairs.
\end{proposition}

By a \emph{prefix} of an element $g\in A_\Gamma$, we mean an element $h\in A_\Gamma$ such that $|g|=|h|+|h^{-1}g|$. In other words, $h$ can be represented by an initial subword of some minimal representative of~$g$.

The \emph{syllable length} $\syl g$ of $g\in A_\Gamma$ is the minimal $k$ such that we can write $g=s_1^{n_1}\dots s_k^{n_k}$. A word $\gamma$ is said to be \emph{syllable-reduced} if it has the form $s_1^{n_1}\dots s_k^{n_k}$ and $k$ is the syllable-length of the element of $A_\Gamma$ that it represents. 


A central object in this paper will be the \emph{extension graph} of a graph, which was introduced by Kim--Koberda \cite{kimkoberda:embedability}. The following is a reformulation of \cite[Def. 1.2]{kimkoberda:embedability}.

\begin{definition}[Extension graph] \label{def:extension_graph}
Let $\Gamma$ be a graph. For each $g\in A_\Gamma$, let $g\cdot\Gamma$ be an isomorphic copy of $\Gamma$, with vertices $\{g\cdot v_i\}$. To obtain the \emph{extension graph} $\Gamma^\ext$ of $\Gamma$ from the disjoint union $\bigsqcup_{g\in A_\Gamma}g\cdot\Gamma$, 
\begin{itemize}
\item   identify vertices $g\cdot v_i\in g\cdot\Gamma$ and $h\cdot v_j\in h\cdot\Gamma$ whenever $gs_ig^{-1}=hs_jh^{-1}\in A_\Gamma$,
\item   then identify each pair of edges with the same endpoints.
\end{itemize}
\end{definition}

For instance, the subgraphs $g\cdot\Gamma$ and $gs_1^n\cdot\Gamma$ of $\Gamma^\ext$ are glued along $\star(g\cdot v_1)$ for all $n\ne0$, because $gs_ig^{-1}=(gs_1)s_i(gs_1)^{-1}$ exactly when $s_i$ commutes with $s_1$. 

More generally, it is a consequence of Proposition~\ref{prop:shuffling} that if $g$ and $h$ are elements of $A_\Gamma$ such that $g\cdot v=h\cdot v$ for some vertex $v\in\Gamma$, then $g$ and $h$ admit minimal representatives that differ by a terminal substring of letters that commute with $s_v$.

Observe that in Definition~\ref{def:extension_graph}, vertices $g\cdot v_i$ and $h\cdot v_j$ can only be identified if $v_i=v_j$.

\begin{definition}[Types and labels]
For $v\in\Gamma$, we say that the vertices of $A_\Gamma\cdot v\subset\Gamma^\ext$ are of \emph{type} $s_v$, and write $\type(v)=s_v$. A \emph{label} of a vertex $w\in A_\Gamma\cdot v$ is any element $g\in A_\Gamma$ such that $g\cdot v=w$. The \emph{minimal label} of $w$ is the label with minimal word-length.
\end{definition}

\begin{remark}
The group $A_\Gamma$ has a natural action on $\Gamma^\ext$. With our convention of using left conjugation, this is a left action given by $g(h\cdot v)=gh\cdot v$. (Note that this differs from \cite{kimkoberda:embedability}.) The action is transitive on $\{g\cdot \Gamma\}$.

As noted above, in order for copies $g_1\cdot \Gamma$ and $g_2\cdot \Gamma$ that make up $\Gamma^\ext$ to intersect, it is sufficient that $g_1$ and $g_2$ differ by a \emph{suffix} that is a power of a single letter. In general, these two copies of $\Gamma$ are not translates of one another by powers of a single generator, because the action is a left action: consider $g_1=s_1s_2$ and $g_2=s_1s_2s_1$, where $s_1$ and $s_2$ are non-commuting standard generators of $A_\Gamma$. This is the same behaviour as the edges of a Cayley graph.

With this comparison to Cayley graphs in mind, one can think of building a ``path of copies of $\Gamma$'' inside $\Gamma^\ext$ by considering increasingly long suffixes. This idea will be fundamental to the construction of Section~\ref{sec:building_qie}.
\end{remark}

According to \cite{kimkoberdalee:finite}, the extension graph of a graph $\Gamma$ can be obtained from $\Gamma$ by repeatedly \emph{doubling}.

\begin{definition}[Double] \label{def:double}
Let $\Gamma$ be a graph, and let $v$ be a vertex of $\Gamma$. The \emph{double} of $\Gamma$ along $v$ is the graph obtained from two disjoint copies $(\Gamma,0)$ and $(\Gamma,1)$ of $\Gamma$ by identifying $(w,0)\sim(w,1)$ for all $w\in\star(v)$.
\end{definition}

The following theorem provides a connection between extension graphs and quasiisometric embeddings. It is a special case of \cite[Cor.~6.7]{baderbensaidpetyt:quasiisometric:rigidity}, which itself relies on the results of \cite{baderbensaidpetyt:from}. A \emph{leaf} of a graph is a vertex of valence at most one. A graph is \emph{square-free} if it does not contain the cycle $C_4$ as an induced subgraph.




\begin{theorem} \label{thm:qie_induces}
Let $\Gamma$ and $\Lambda$ be triangle-free graphs. If $\Gamma$ has no leaves and $\Lambda$ is square-free, then every quasiisometric embedding $A_\Gamma\to A_\Lambda$ induces a combinatorial embedding $\Gamma^\ext\to\Lambda^\ext$.
\end{theorem}

Note that the conditions of \cref{thm:qie_induces} are satisfied if $\Gamma$ and $\Lambda$ are both cycles of length at least five.

\section{From extension graphs to homomorphisms} \label{sec:induced_implies_subgroup}

Let $\Gamma$ and $\Lambda$ be finite graphs. In \cite[Thm~1.3]{kimkoberda:embedability}, Kim--Koberda show that if $\iota:\Gamma\hookrightarrow\Lambda^\ext$ is an embedding with induced image, then $A_\Gamma$ can be found as a subgroup of $A_\Lambda$. Moreover, the proof of \cite[Cor~1.15]{kimkoberda:embedability} shows that $A_\Gamma$ can be found as an undistorted subgroup of $A_\Lambda$. They give three proofs for this, with two going through mapping class groups. The most elementary argument, though, is a combination of two observations:
\begin{itemize}
\item   If $\Lambda'$ is the double of $\Lambda$ along some vertex $w$, then $A_{\Lambda'}$ is isomorphic to the kernel of the map $A_\Lambda\to\sfrac\Z{2\Z}$ sending the generator $t_w$ to 1 and all other generators to 0 \cite[Ex.~1.4]{bestvinakleinersageev:asymptotic}.
\item   Every finite induced subgraph of $\Lambda^\ext$ is contained in an induced subgraph of $\Lambda^\ext$ obtained by a finite sequence of doubles \cite[Lem.~3.1]{kimkoberda:embedability}.
\end{itemize}
Together, these observations show that $A_\Gamma$ is isomorphic to a parabolic subgroup of a finite-index RAAG subgroup of $A_\Lambda$, and hence is undistorted.

Note, however, that though the subgroup embedding induces a map $\Gamma^\ext\hookrightarrow\Lambda^\ext$, the embedding $\iota$ may not be recoverable from that map. We address this with the following refinement of \cite[Thm~1.3]{kimkoberda:embedability}.



\begin{proposition} \label{prop:induced_to_homo}
Let $\Gamma$ and $\Lambda$ be finite graphs. For any embedding $\iota:\Gamma\to\Lambda^\ext$ with induced image, there is an injective homomorphism $F\colon A_\Gamma\to A_\Lambda$ that is a quasiisometric embedding and induces an embedding of $\Gamma^\ext$ in $\Lambda^\ext$ as an induced subgraph, sending $1\cdot\Gamma$ to $\iota(\Gamma)$.
\end{proposition}


\begin{proof}
We first consider the case where $\iota(\Gamma)\subseteq 1\cdot\Lambda$.
In this case, $\iota(\Gamma)$ corresponds to a parabolic subgroup of $A_\Lambda$,
which is undistorted. Equivalently, the map
\[
F(s_v)=\type(\iota(v))
\]
extends to a homomorphism $A_\Gamma\to A_\Lambda$ that is a quasiisometric embedding.

In general, by \cite[Lem.~3.1]{kimkoberda:embedability}, the induced
subgraph $\iota(\Gamma)$ is contained in an induced subgraph $\Lambda'$ of
$\Lambda^\ext$ obtained from $1\cdot\Lambda$ by taking finitely many doubles. Thus, by the argument above, there exists an undistorted homomorphism $F':A_\Gamma\to A_{\Lambda'}$ inducing $\iota$. It therefore remains
to show that, whenever $\Lambda'$ is obtained from $\Lambda$ by finitely many
doubles, the corresponding inclusion can be realised by a
homomorphism $A_{\Lambda'}\to A_\Lambda$ that is a quasiisometric embedding. Since a finite sequence of doubles
can be handled by induction, it suffices to treat the case of a single double.

Thus, assume that $\Gamma$ is the double of $\Lambda$ along a vertex $v_0$.
More precisely, write
\(
\Gamma=\Gamma_1\cup_{\star(v_0)}\Gamma_2,
\)
where $\Gamma_1$ and $\Gamma_2$ are copies of $\Lambda$ glued along
$\star(v_0)$. We may assume that
\(
\iota(\Gamma_1)=1\cdot\Lambda\) and \(\iota(\Gamma_2)=t_0^n\cdot\Lambda,
\)
where $t_0=t_{w_0}$ is a standard generator of $A_\Lambda$,
$n\in\mathbb Z\ssm\{0\}$, and $\iota(v_0)=w_0$. We may assume $n\geq 1$, by applying the isomorphism taking $t_0$ to its inverse and fixing all the other standard generators of $A_\Lambda$.

Inspired by \cite[Ex.~1.4]{bestvinakleinersageev:asymptotic}, define a map $F\colon A_\Gamma\to A_\Lambda$ by declaring 
\[
F(s_v) = \begin{cases}
    t_0^{n+1} & \text{if } v=v_0 \\
    \type(\iota(v)) & \text{if } v\in\Gamma_1\ssm\{v_0\} \\
    t_0^n\type(\iota(v))t_0^{-n} & \text{if } v\in\Gamma_2\ssm\Gamma_1.
    \end{cases}
\]
    Recall that $s_v$ is the generator of $A_{\Gamma}$ corresponding to $v$.
It is easy to see that $F$ extends to a homomorphism. 

An equivalent way of writing this is the following. We have isomorphisms $i_1\colon A_{\Gamma_1}\to A_\Lambda$ and $i_2\colon A_{\Gamma_2}\to A_\Lambda$ induced by the graph isomorphisms. The RAAG $A_\Gamma$ is identified with the amalgam $A_{\Gamma_1}*_{A_{\str(v_0)}}A_{\Gamma_2}$ and the isomorphisms $i_1$ and $i_2$ agree on the intersection. We then define 
\[F(s_v)=\begin{cases}
    t_0^{n+1} & \text{if } v=v_0 \\
    i_1(s_v) & \text{if } v\in\Gamma_1\ssm\{v_0\} \\
    t_0^ni_2(s_v)t_0^{-n} & \text{if } v\in\Gamma_2\ssm\Gamma_1.
    \end{cases}\]

It is clear that the map $F$ induces the desired map on extension graphs.
We must show that $F$ is a quasiisometric embedding and is injective, which is not necessarily obvious from its definition. Since group homomorphisms from finitely generated groups are always Lipschitz, it suffices to show that $F$ is colipschitz. This is equivalent to showing that $|F(g)|$ is linearly lower-bounded by $|g|$ for all $g\in A_\Gamma$, because $\dist(F(g),F(h))=\dist(1,F(g^{-1}h))$.

Let $g\in A_\Gamma$. Using the decomposition of $A_\Gamma$ coming from the
double, write
\[
g=g_1g_2\cdots g_k,
\]
as a reduced and syllable-reduced expression,
where $g_i\in A_{\Gamma_1}$ for $i$ odd, where
$g_i\in A_{\Gamma_2\ssm\Gamma_1}$ for $i$ even, and where $g_i\neq 1$ for $i\geq 1$. (We allow $g_1=1$ and still consider the expression reduced.)
Let $s_0=s_{v_0}$ be the generator of $A_\Gamma$ corresponding to $v_0$.
For $i$ odd, we write $g_i=s_0^{m_i}g_i's_0^{m_i'}$ where $m_i,m_i'\in\mathbb{Z}$ and $g_i'\in A_{\Gamma_1}$ does not have $s_0^{\pm 1}$ as a prefix or suffix. If $g_i'=1$, then we assume $m_i'=0$. The proposition will follow from the following claim, which essentially says that the only cancellations in $F(g_1)\cdots F(g_k)$ can occur in the $t_0$--syllables, and there they are controlled.

\begin{claim*}
For $g=g_1\cdots g_k$ as above, 
    $F(g)=t_0^{\ell_0}i_1(g_1')t_0^{\ell_1}i_2(g_2)t_0^{\ell_2}\cdots$ is a reduced and syllable-reduced expression in $A_\Lambda$ where:
    \begin{itemize}
        \item $\ell_0=m_1(n+1)$ if $g_1'\neq 1$ and $\ell_0=m_1(n+1)+n$ if $g_1'=1$;
        \item if $g_i'=1$ for some $1<i<k$ odd then $\ell_{i-1}=m_i(n+1)$ and $\ell_i=0$;
        \item if $g_i'\neq 1$ for $1<i<k$ then $\ell_{i-1}=(n+1)m_i'+n$ for $i$ odd and $\ell_i=(n+1)m_i-n$.
        \item $\ell_k=(n+1)m_k'$ if $k$ is odd and $\ell_k=-n$ if $k$ is even.
    \end{itemize}
\end{claim*}

\begin{claim*proof}
    We prove by induction on $k$. For $k=1$ the claim is clear.
    Assume the claim is true for $k-1$. We divide into cases by the parity of $k$.

    If $k$ is even, then by induction $F(g_1\cdots g_{k-1})=t_0^{\ell_0}i_1(g_1')t_0^{\ell_1}i_2(g_2)t_0^{\ell_2}\cdots i_1(g_{k-1}') t_0^{m_k'(n+1)}$ is reduced and syllable-reduced. By definition,  $F(g_k)=t_0^ni_2(g_k)t_0^{-n}$, where $g_k\neq 1$ and no letter of $i_2(g_k)$ commutes with $t_0$. We note that $m_k'(n+1)+n\neq 0$ and so the only cancellations that may occur are from the letters of $i_2(g_k)$, but none of them commutes with $t_0$, since $i_2$ is an isomorphism and none of the letters of $g_k$ commute with $s_0$. Thus the expression $t_0^{\ell_0}i_1(g_1')t_0^{\ell_1}i_2(g_2)t_0^{\ell_2}\cdots i_1(g_{k-1}') t_0^{m_k'(n+1)+n}i_2(g_k)t_0^{-n}$ is reduced.

    If $k$ is odd, then by induction $F(g_1,\cdots g_{k-1})=t_0^{\ell_0}i_1(g_1')t_0^{\ell_1}i_2(g_2)t_0^{\ell_2}\cdots i_2(g_{k-1}) t_0^{-n}$ is reduced and syllable-reduced. By definition $F(g_k)=t_0^{m_k(n+1)}i_1(g_k')t_0^{m_k'(n+1)}\neq 1$ and it is reduced and syllable-reduced. As in the previous case, since $m_k(n+1)+n\neq 0$, the only cancellation can occur from a letter of $i_2(g_k)$. In other words, there is a generator $t\neq t_0$ of $A(\Lambda)$ that commutes with $t_0$ such that $t^{\pm 1}$ is a prefix of $i_2(g_k)$ and $t^{\mp 1}$ is a suffix of $F(g_1\cdots g_{k-1})$. But then $i_1^{-1}(t)=i_2^{-1}(t)$ will be cancelled in the expression $g_1\cdots g_k$, in contradiction to it being reduced.
\end{claim*proof}

By an easy calculation, we see that $|F(g)|\geq |g|$ and so the proposition is proved.
\end{proof}

In order to motivate the more difficult constructions of \cref{sec:building_qie}, we now explain how one can come up with the homomorphism $F$ in the above proof by looking at extension graphs. 

Let $\Lambda$ be a graph and let $\Gamma=\Gamma_1\cup_{\star(v_0)}\Gamma_2$ where $\Gamma_1$ and $\Gamma_2$ are copies of $\Lambda$. Let $\iota:\Gamma\to\Lambda^\ext$ be an injective map, where $\iota(\Gamma_1)=1\cdot\Lambda$ and $\iota(\Gamma_2)=t_0^n\cdot\Lambda$ where $w_0=\iota(v_0)$ is the vertex corresponding to $t_0$.
We will extend $\iota$ to an embedding $\Gamma^\ext\to\Lambda^\ext$, and use that embedding to construct a ``route map'' $A_\Gamma\to A_\Lambda$, inspired by \cite[\S7.1]{baderbensaidpetyt:quasiisometric:rigidity}. 

\bsh{Blocks}
We refer to $\iota:\Gamma\to (1\cdot\Lambda)\cup(t_0^n\cdot\Lambda)$ as a ``block'' $B(1)$ (we imagine the block as a copy of $\Gamma$ living inside $\Lambda^\ext$, keeping track of the data of how it is embedded). For $h\in A_\Lambda$, we define $B(h)$ to be the map $\Gamma\to\Lambda^\ext$ given by $v\mapsto h\cdot \iota(v)$, whose image is $(h\cdot \Lambda)\cup(ht_0^n\cdot \Lambda)$. 
\esh

Since $\Gamma^\ext$ is obtained from $\Gamma$ by repeatedly doubling, our strategy for finding $\Gamma^\ext$ inside $\Lambda^\ext$ will be to find a way to mimic the doubling procedure inside $\Lambda^\ext$ with blocks instead of copies of $\Lambda$.


\bsh{Doubling}
Each vertex of the block $B(h)$ can be written as $ht_0^\epsilon\cdot w$ for some $w\in\Lambda$ and some $\epsilon\in\{0,n\}$. For an integer $a$, we define the \emph{$a^\mathrm{th}$ double} of the block $B(h)$ along $ht_0^\epsilon\cdot w$ to be 
\[
\doub_{ht_0^\epsilon\cdot w}^a(B(h)) \,=\, B(ht_0^\epsilon t_w^at_0^{-\epsilon}).
\]
It is worth thinking about how the double of $B(h)$ intersects $B(h)$ for each of the two values of $\epsilon$.
\esh

We will use this doubling map for reconstructing $\Gamma^\ext$ inside $\Lambda^\ext$, but only in the case when $w\ne w_0$. The issue when $w=w_0$ is that the $n^\mathrm{th}$ double of $B(h)$ along $h\cdot w_0$ intersects $B(h)$ in a copy of $\Lambda$. We handle this by using a different operation that we call \emph{gliding}. The name ``glide'' makes more sense in the context of \cref{sec:building_qie}, where the operation looks like a glide. We are using the same name here to try to show the similarities.

\bsh{Gliding}
For an integer $a$, we define the \emph{$a^\mathrm{th}$ glide} of $B(h)$ to be 
\[
\glide_{h\cdot w_0}^a(B(h)) \,=\, B(ht_0^{a(n+1)}).
\] 
Note that the image of the glide of $B(h)$ is $(ht_0^{a(n+1)}\cdot\Lambda)\cup(ht_0^{a(n+1)+n}\cdot\Lambda)$, which intersects $B(h)$ only in the star of $h\cdot w_0$.
\esh

We can now use doubling and gliding to define a combinatorial embedding $f\colon\Gamma^\ext\to \Lambda^\ext$. 

\bsh{Reconstructing extension graphs}
We define $f|_{1\cdot \Gamma}=\iota$ and proceed by induction on syllable length. Assume that we have defined $f$ on $g\cdot \Gamma$ for all $g\in A_\Gamma$ with $\syl g\leq d$ in such a way that $f|_{g\cdot \Gamma}$ is a block. Given $g\in A_\Gamma$ with $\syl g = d+1$, we can write $g=g's_v^a$ for some $g'\in A_{\Gamma}$ with $\syl g=d$ and some $v\in\Gamma$. 
If $v\neq v_0$, then we define 
\[
f(g\cdot\Gamma) \,=\, \doub^a_{f(g'\cdot v)}(f|_{g'\cdot\Gamma}).
\]
If instead $v=v_0$, then we define
\[
f(g\cdot\Gamma) \,=\, \glide^a_{f(g'\cdot v_0)}(f|_{g'\cdot\Gamma}).
\]
\esh

One can check that $f$ is well defined. This follows from the fact that doubling/gliding along adjacent vertices commutes. One can also directly check that $f$ is an embedding, or (as we will do in Section~\ref{sec:building_qie}) deduce it from the existence of a quasiisometric embedding $F\colon A_\Gamma\to A_\Lambda$ that induces $f$, which we shall now construct.

\bsh{The lifted map $F$}
We define $F$ to be the ``route map'' that, for $g\in A_{\Gamma}$, outputs the origin cell of the image of the block $f(g\cdot\Gamma)$. That is, $F(g)=h$, where $f|_{g\cdot\Gamma}=B(h)$. 
\esh

Now the only thing left to show is that $F$ is a quasiisometric embedding. But it is easy to see that this map we constructed is the same as the map in \cref{prop:induced_to_homo} and so we are done.

In the next section, in order to define a quasiisometric embedding we will first define a combinatorial embedding of extension graphs and take its ``route map'', sending an element to one of the labels in the block corresponding to its image. Then we will have to work slightly harder than in \cref{prop:induced_to_homo} in order to show it is in fact a quasiisometric embedding.

\section{Constructing quasiisometric embeddings between cycle RAAGs} \label{sec:building_qie}

In this section we prove the bulk of \cref{mthm:cycles_qie}. We construct quasiisometric embeddings $A_{C_m}\to A_{C_n}$ when $m=n+p(n-4)+q(n-2)$ and $p>0$.

When $n=4$ these were constructed by Rull in \cite{rull:embedding} and when $n=5$ or $n=6$, \cite[Thm~1.12, Cor.~1.15]{kimkoberda:embedability} show that there exist homomorphisms that are quasiisometric embeddings. Thus we will assume $m,n>6$, and fix a way of writing $m=n+p(n-4)+q(n-2)$ with $p>0$. When $q>0$ our construction furthermore yields a quasiisometric embedding that is not at finite distance from a homomorphism, even when there is known to be a quasiisometric embedding that is close to a homomorphism, see \cref{lem:no_homo}.

\begin{remark} \label{rem:n_more_than_six}
When $q=0$ our construction will, up to finite distance, yield a homomorphism (see \cref{lem:homo}), and thus we recover one direction of \cite[Thm 1.12, Cor. 1.15]{kimkoberda:embedability} for $n>6$. The proof of \cite[Thm 1.12, Cor. 1.15]{kimkoberda:embedability} constructs the homomorphism using embeddings in mapping class groups; we do not require that.
    
The authors do not believe that the assumption $n>6$ is fundamentally necessary: there should be modifications of the arguments that work for these values of $n>4$ and construct quasiisometric embeddings that are not at finite distance from a homomorphism when $q>0$.
\end{remark}

To improve readability, let us write $\Gamma=C_m$ and $\Lambda=C_n$. Write $\{v_i\,:\,i\in\sfrac{\Z}{m\Z}\}$ and $\{w_j\,:\,j\in\sfrac{\Z}{n\Z}\}$ for the vertex sets of $\Gamma$ and $\Lambda$, respectively. Denote the standard generator of $A_\Gamma$ corresponding to the vertex $v_i$ by $s_i$, and the standard generator of $A_\Lambda$ corresponding to $w_j$ by $t_j$. Write $S=\{s_i\}$ and $T=\{t_j\}$. 

With the vertices identified this way, there are two natural maps that we shall need to consider. Firstly, we have the \emph{shift} map $\sigma$ on $\Gamma=C_m$, given by $\sigma(v_i)=v_{i+1}$. Secondly, we consider a map which is essentially an extension of a shift map to $\Lambda^\ext$. Namely, we define the \emph{rotation} map $\rho:\Lambda^\ext\to\Lambda^\ext$ by setting

\[
\rho(t_{k_1}^{a_1}t_{k_2}^{a_2}\dots t_{k_q}^{a_q}\cdot w_j) \,=\, t_{k_1+1}^{a_1}t_{k_2+1}^{a_2}\dots t_{k_q+1}^{a_q}\cdot w_{j+1}.
\]
If we think of $\Lambda^\ext$ as having a ``central'' copy of $\Lambda$, from which the rest is obtained by repeatedly doubling, then the map $\rho$ rotates $\Lambda^\ext$ about the centre of that copy of $\Lambda$.

\bsh{Blocks} \label{sh:blocks} 
Let $h\in A_\Lambda$. We specify two subgraphs $B^+(h)$ and $B^-(h)$ of $\Lambda^\ext$, which we refer to, respectively, as the \emph{positive} and \emph{negative} \emph{basic blocks emanating from $h$}. Recall that $p>0$ and $q\ge0$ are fixed integers such that $m=n+p(n-4)+q(n-2)$.
For $\eps\in\{+,-\}$, the basic block $B^\eps(h)$ is built as follows.
\begin{itemize}
\item   Let $B^\eps_0(h)=h\cdot \Lambda\subseteq\Lambda^\ext$. 
\item   For $i\in\{0,\dots,p-1\}$, let $B^\eps_{i+1}(h)=t_{\eps2(-1)^i}\cdot B^\eps_i(h)$.
\item   For $j\in\{p,\dots,p+q-1\}$, let $B^\eps_{j+1}(h)=t_{\eps(-1)^j}t_{\eps2(-1)^j}\cdot B^\eps_j(h)$.
\item   We define $B^\eps(h)=\bigcup_{k=0}^{p+q} B^\eps_k(h)$.
\end{itemize}
See Figure~\ref{fig:basic_blocks}.
When $h=1$, we sometimes write $B^\eps=B^\eps(1)$ as in Figure~\ref{fig:basic_blocks}.

The term \emph{block} will refer to a subgraph of $\Lambda^\ext$ of the form $\rho^\kappa(B^\eps(h))$ for some $\kappa\in\Z$, $\eps\in\{+,-\}$, $h\in A_\Lambda$. That is, a block is something obtained by starting at some $h\cdot \Lambda$, building a basic block from there, and then rotating some number of times. We refer to $\eps$ as the \emph{sign} of the block $\rho^\kappa(B^\eps(h))$.

Given a block $B=\rho^\kappa(B^\eps(h))$ and a vertex $w\in B$, we define the \emph{depth} of $w$ inside $B$ to be 
\[
\dep(w;B) \,=\, \min\{i\,:\,\rho^{-\kappa}(w)\in B^\eps_i(h)\}.
\]
In other words, the depth of $w$ is the first time at which $\rho^{-\kappa}(w)$ appears in the construction of $B^\eps(h)$.

\begin{figure}[ht]
\begin{center} \makebox[0pt]{\begin{minipage}{1.4\textwidth}
\includegraphics{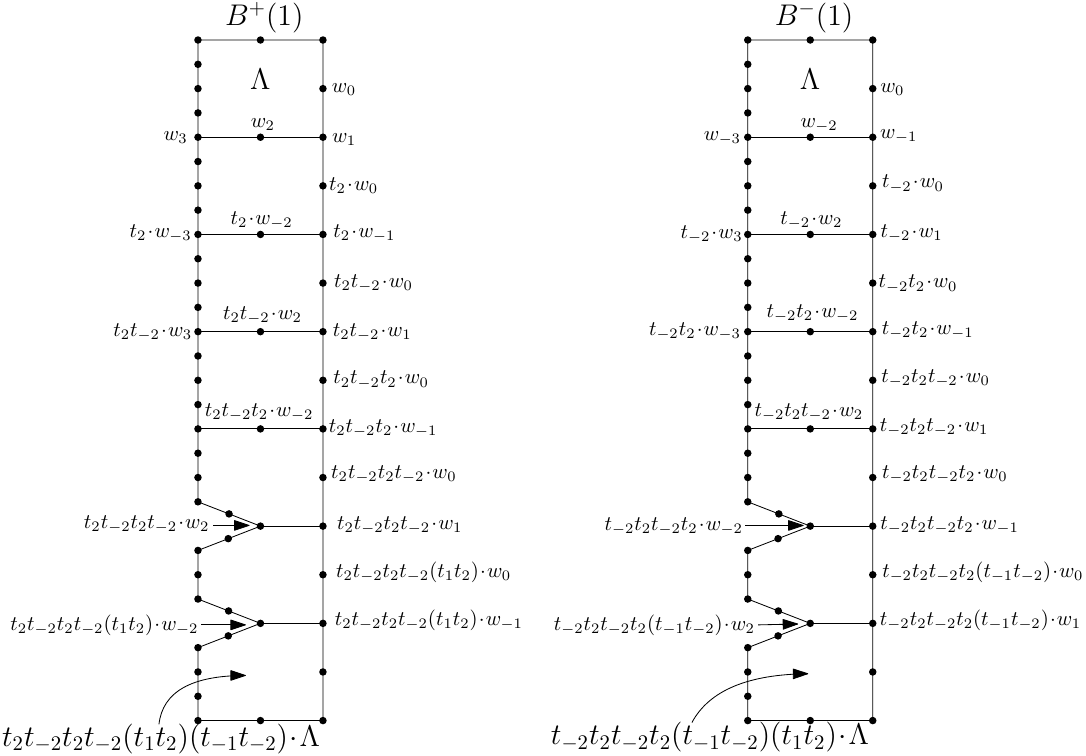}\centering 
\caption{The two basic blocks emanating from 1 in $\Lambda^\ext$, illustrated with $p=4$ and $q=2$.} \label{fig:basic_blocks}
\end{minipage}}\end{center}
\end{figure}
\esh

\bsh{Marked blocks} \label{sh:marked_blocks}
By construction, the basic blocks $B^\eps(h)$ have natural boundary paths of length $n+p(n-4)+q(n-2)=m$. We can therefore fix identifications $\iota_{\eps,h}:C_m\to\partial B^\eps(h)$. 
We do this by declaring $\iota_{\eps,h}(v_0)=h\cdot w_0$, and $\iota_{+,h}(v_1)=h\cdot w_1$, and $\iota_{-,h}(v_1)=h\cdot w_{-1}$, then extending cyclically. For example, $\iota_{+,1}(v_3)=t_2\cdot w_{-1}$ and $\iota_{-,1}(v_3)=t_{-2}\cdot w_1$, as can be seen in Figure~\ref{fig:basic_blocks}. Note that $\iota_{+,h}$ and $\iota_{-,h}$ orient $h\cdot\Lambda$ oppositely.

By a \emph{marked block}, we mean a map $\Gamma\to\Lambda^\ext$ of the form $v_i\mapsto\rho^\kappa\iota_{\eps,h}\sigma^\alpha(v_i)$, where $\sigma:C_m\to C_m$ is the shift map and $\rho:\Lambda^\ext\to\Lambda^\ext$ is the rotation map. In other words, it is a subgraph of $\Lambda^\ext$ obtained from $C_m$ by rotating, identifying with a basic block, and adding $\kappa$ to all subscripts. We denote this marked block by $B(\alpha,\eps,\kappa;h)$, to emphasise that we think of it as being a copy of $B(\alpha,\eps,\kappa;1)$ that emanates from $h$. 
\esh

\bsh{Labels in blocks} \label{sh:block_labels}
Let $B=B(\alpha,\eps,\kappa;h)$ be a marked block. From the construction of $B^\eps(h)$ in Item~\ref{sh:blocks}, we see that we can write explicit expressions for labels of vertices in $B$. For instance, as shown in Figure~\ref{fig:labels}, if $p=2$ and $q=1$, then the vertices of $B(\alpha,+,3;h)$ can all be labelled by an element of 
\[
\{\, h,\, ht_5,\, ht_5t_1,\, ht_5t_1t_4t_5\, \}.
\]
Note that these expressions may not give minimal labels, for instance if $h=t_5^{-1}$, or if we are considering a vertex adjacent to, say $ht_5t_1\cdot w_4$.

\begin{figure}[ht]
\centering\includegraphics[height=8.5cm]{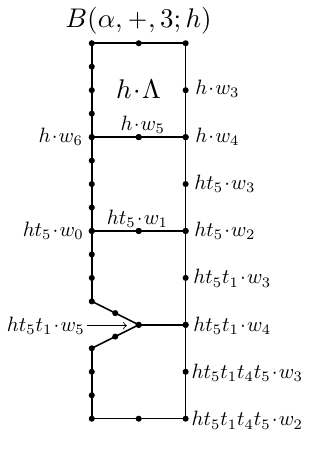}
\caption{The block $B(\alpha,+,3;h)$. The labels are given by Item~\ref{sh:blocks}, but note that $ht_5t_1t_4t_5\cdot w_3=ht_5t_1t_5\cdot w_3$, because $t_4$ commutes with both $t_3$ and $t_5$, so the given label is not the minimal element of $A_\Lambda$ that could be used to label that vertex.} \label{fig:labels}
\end{figure}

We call the set of labels arising in the construction of the blocks $\{B(0,\eps,\kappa;1)\}_{\epsilon\in\{+,-\},\kappa\in\mathbb{Z}}$ \emph{basic labels}. Let us write $\L'(\eps,\kappa)\subseteq A_\Lambda$ for the subset of basic labels coming from $B(0,\eps,\kappa;1)$. It has cardinality $1+p+q$. By construction, the vertices of the block $B(\alpha,\eps,\kappa;h)$ can all be labelled by a word of the form $h\ell$, where $\ell\in\L'(\eps,\kappa)$.

Informally, basic labels tell us how to move from the copy of $\Lambda$ at the ``top'' of a block to other copies of $\Lambda$ with greater depth. In the constructions below, we will want to move between two copies of $\Lambda$ that make up a block, without assuming that either has depth zero. For this we shall use \emph{basic relative labels}.

By a \emph{basic relative label}, we mean an element of the form $\ell^{-1}\ell'$, where $\ell,\ell'\in\L'(\eps,\kappa)$ for some $\eps,\kappa$. We write $\L(\eps,\kappa)$ for the set of such basic relative labels. Note that $\L'(\eps,\kappa)\subset\L(\eps,\kappa)$, because $1\in\L'(\eps,\kappa)$. When we are considering a block $B$ for which the values of $\eps$ and $\kappa$ have not been made explicit, we drop them from the notation and simply write $\L$.
\esh

\bsh{Shortcuts} \label{sh:shortcuts}
If $q>0$, then $\iota_{\eps,h}(\Gamma)=\partial B^\eps(h)$ is not an induced subgraph of $\Lambda^\ext$: there are exactly $q$ shortcuts, given by the final $q$ stages of the construction of $B^\eps(h)$. More precisely, for each $i\in\{p,\dots,p+q-1\}$, there is a pair of vertices $(w,w')\in A_\Lambda\cdot \{(w_1,w_2),(w_{-1},w_{-2})\}$ such that: both $w$ and $w'$ have depth $i$ inside $B^\eps(h)$; we have $\{w,w'\}=B^\eps_i(h)\cap B^\eps_{i+1}(h)$; and the $\iota_{\eps,h}$--preimages of $w$ and $w'$ are not adjacent.

Given a block $B=\rho^\kappa(B^\eps(h))$, let us write $\short(B)$ for the set of vertices $w\in B$ such that $\rho^{-\kappa}(w)$ is a vertex appearing in one of the $q$ pairs as above. Referring again to Figure~\ref{fig:basic_blocks}, let us say that $w\in\short(B)$ is \emph{early} if $\rho^{-\kappa}(w)\in A_\Lambda\cdot \{w_1,w_{-1}\}$, and \emph{late} if $\rho^{-\kappa}(w)\in A_\Lambda\cdot \{w_2,w_{-2}\}$. Intuitively, the early shortcut vertices are the ones on the ``less dense'' side of the block. We refer to two adjacent vertices of $B$ that form a shortcut as a \emph{shortcut pair}, and each vertex is the \emph{shortcut twin} of the other.

Note that, by the construction of the basic blocks, for every marked block $B$, no two vertices of $\short(B)$ are images of adjacent vertices of $\Gamma$. Furthermore, note that no vertex of $h\cdot \Lambda\subseteq B^\eps(h)$ is adjacent to any shortcut of $B^\eps(h)$, because we assume $p >0$.

We refer to a pair of vertices of a block $B$ that are adjacent but do not form a shortcut as being \emph{boundary-adjacent}.
\esh

Our next goal is to describe a map $f\colon\Gamma^\ext\to\Lambda^\ext$. Bearing in mind that $\Gamma^\ext$ is obtained by making identifications in $\bigsqcup_{g\in A_\Gamma}g\cdot \Gamma$, the construction of $f$ will be done by an iterative procedure. The procedure involves two operations for transporting marked blocks inside $\Lambda^\ext$, which we shall refer to as \emph{doubling} and \emph{gliding}.

\bsh{Doubling} \label{sh:doubling}
The doubling operation is simply an extension to blocks of the standard way to double copies of $\Lambda$ inside $\Lambda^\ext$, as described in \cite{kimkoberda:embedability}: the standard double of $h\cdot \Lambda$ along $h\cdot w_j$ is $ht_j\cdot \Lambda$. Note that $h$ may not be the minimal label of $h\cdot w_j$, and there is an analogous situation for blocks.

Let $B=B(\alpha,\eps,\kappa;h)$ be a marked block, and let $w\in B$. We can write $w=h\ell\cdot w_j$ with $\ell\in\L(\eps,\kappa)$. For a nonzero integer $a$, we define the \emph{$a^\mathrm{th}$ double} of $B$ along $w$ to be
\[
\doub^a_w(B) \,=\, B(\alpha,\eps,\kappa;h\ell t_j^a\ell^{-1}).
\]
This is illustrated in Figure~\ref{fig:double}. Note that the resulting block here does not depend on the choice of $\ell\in\L(\eps,\kappa)$, if there is more than one option.

\begin{figure}[ht]
\includegraphics[width=16cm]{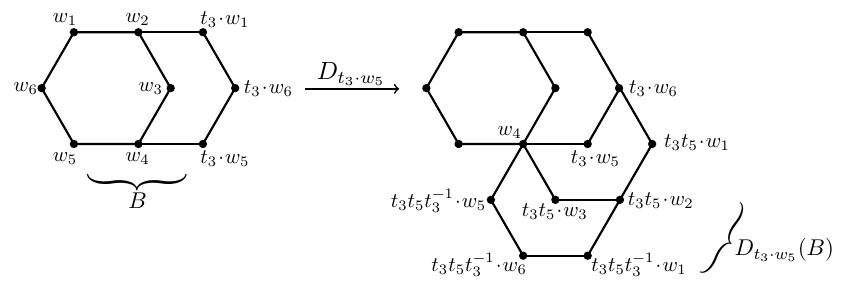}\centering 
\caption{Example of the doubling operation in $C_6^\ext$, illustrated with $B=B(\alpha,+,1;1)$.} \label{fig:double}
\end{figure}
\esh

\begin{lemma} \label{lem:doubling_commutes}
If $w$ and $w'$ are adjacent vertices of a marked block $B$, then $\doub^a_w\doub^{a'}_{w'}(B)=\doub^{a'}_{w'}\doub^a_w(B)$.
\end{lemma}

\begin{proof}
Since $(w,w')$ is an edge of $\Lambda^\ext$, there exists $h\in A_\Lambda$ such that $w=h\cdot w_j$ and $w'=h\cdot w_k$, where $(w_j,w_k)$ is an edge of $\Lambda$, and hence $t_j$ and $t_k$ commute. The lemma follows from an easy computation.
\end{proof}

Unlike the doubling operation, we do not define the gliding operation for all vertices in all blocks. Moreover, it does not behave like a morphism, and we instead define explicitly what should be thought of as its powers. The role of gliding will be to act as a ``fake doubling'' operation in the construction of the map $f$ in situations where we cannot use the actual doubling operation.

\bsh{Gliding} \label{sh:gliding}
Let $B=B(\alpha,\eps,\kappa;h)\subseteq\Lambda^\ext$ be a marked block, and let $w\in B$. Let $\ell\in\L(\eps,\kappa)$ be minimal such that $w=h\ell\cdot w_j$ for some $j$. By the definition of the marked block $B$, there is some $i$ such that $w=\rho^\kappa\iota_{\eps,h}\sigma^\alpha(v_i)$. Let $\delta=\dep(w;B)$. We shall only define glides along $w$ in the case where $w\in\short(B)$, and we divide into two cases according to whether $w$ is an early or late shortcut of $B$. See Figures~\ref{fig:early_glide} and~\ref{fig:late_glide}, respectively. Recall from Item~\ref{sh:shortcuts} that the early shortcuts correspond to the vertices ``on the side of $B$ with fewer vertices'', which come earlier in the basic marked blocks. Regardless of whether $w$ is early or late, we necessarily have $\ell\ne1$, because $p>0$.

For a nonzero integer $a$, let $\hat a=a+1$ if $a>0$, and let $\hat a=a$ if $a<0$. If $w=\rho^\kappa\iota_{\eps,h}\sigma^\alpha(v_i)$ is an early shortcut of $B$, then define the \emph{$a^\mathrm{th}$ glide} of $B$ along $w$ to be
\[
\glide^a_{h\ell\cdot w_j}(B(\alpha,\eps,\kappa;h)) \,=\, B(1-i,\eps(-1)^\delta,\kappa;h\ell t_j^{\hat a}).
\]
We call this an \emph{early glide}. See Figure~\ref{fig:early_glide}. On the other hand, if $h\ell\cdot w_j$ is a late shortcut of $B$, then define
\[
\glide^a_{h\ell\cdot w_j}(B(\alpha,\eps,\kappa;h)) \,=\, B(3-n-i,\eps(-1)^\delta,\kappa-\eps(-1)^\delta;h\ell t_j^{\hat a}),
\]
where we recall that $\Lambda=C_n$. We call this a \emph{late glide}. See Figure~\ref{fig:late_glide}.

The idea of the glide operation is to overlay a block $B'$ on top of the shortcut vertex $w\in B$ that we are gliding along, in such a way that the intersection of $B'$ with $B$ is precisely the star of $w$. There are many options for how to overlay a block in that way, but we fix a systematic way: $w$ will lie in both $\partial B'$ and the intersection of the first two copies of $\Lambda$ comprising $B'$, with the vertex $w'$ that it forms a shortcut with lying in $B'\ssm\partial B'$. Moreover, if $w$ is an early shortcut then we match it with the early side of $B'$, and if it is late then we match it with the late side of $B'$. This leads to the above formulas as follows. 

Referring to Figure~\ref{fig:basic_blocks}, if we read the vertices of $B^+$ clockwise and consider the intersections of the copies of $\Lambda$ comprising $B^+$, then note that they alternately count up or down, whereas in $B^-$ they alternately count down or up. Thus the sign of $B'$ depends on the sign of $B$ and on the depth of $w$ inside $B$. The amount we need to rotate by is given by comparing the type of the vertex $w$ with either $\eps$ or $\eps(n-3)$, according to the sign of $B'$ and whether $w$ is early or late. Replacing $a$ by $\hat a$ ensures that $B$ and $B'$ intersect in no more than the star of $w$. Finally, we wish the marked blocks $B$ and $B'$ to restrict to the same marking on $w$ and its two neighbours in the loop $\partial B$, which we accomplish with an appropriate shift of $\Gamma$.
\esh


\begin{figure}[ht]
\centering\includegraphics[height=12cm, trim= 0 5mm 0 5mm]{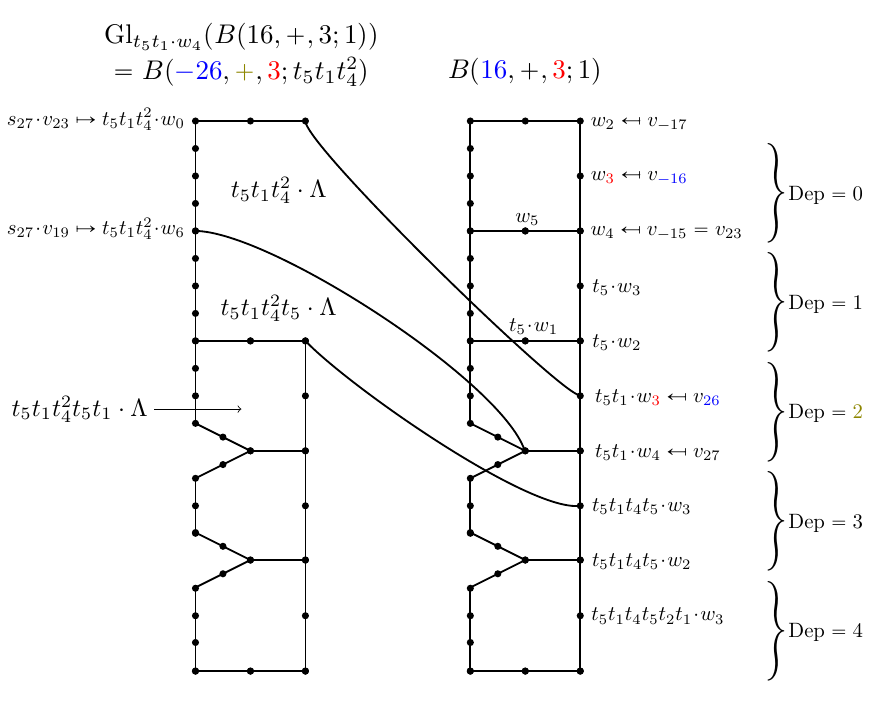}
\caption{An early glide, illustrated with $m=38$, $n=10$. The glide is along the vertex $t_5t_1\cdot w_4$, which has depth two. The origin of the marking of the new block is highlighted.} \label{fig:early_glide}
\end{figure}

\begin{figure}[ht]
\centering\includegraphics[height=12cm, trim= 0 5mm 0 5mm]{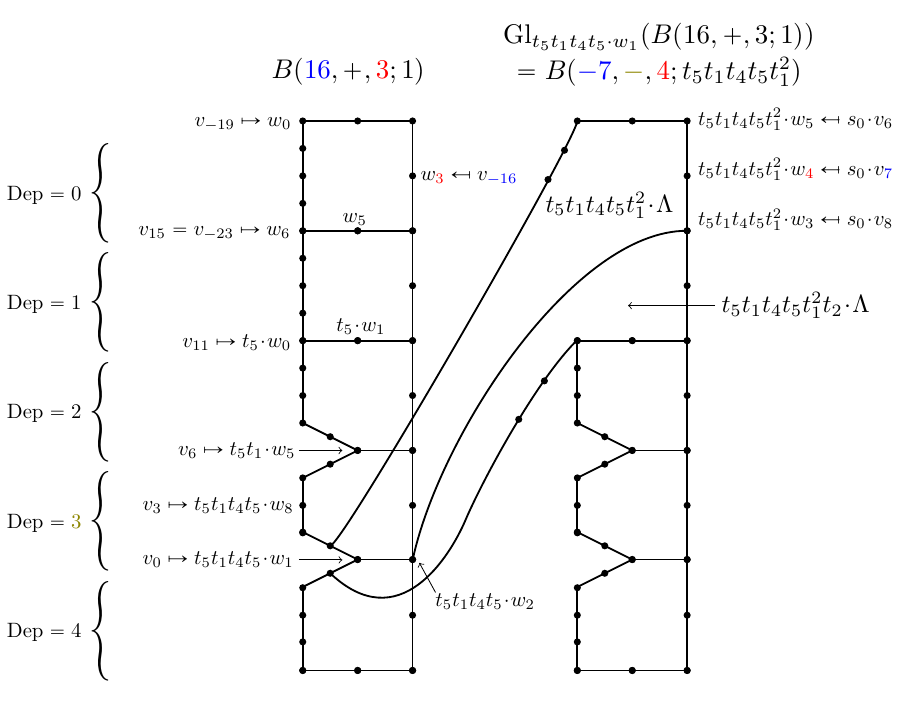}
\caption{A late glide, illustrated with $m=38$, $n=10$. The glide is along the vertex $t_5t_1t_4t_5\cdot w_1$, which has depth three. The origin of the marking of the new block is highlighted.} \label{fig:late_glide}
\end{figure}

Having defined the doubling and gliding operations, we make note of a few of their basic properties.

\begin{lemma} \label{lem:stay_shortcut}
If $w$ and $w'$ are vertices of a marked block $B$, then $w'\in\short(B)$ if and only if $\doub^a_w(w')\in\short(\doub^a_w(B))$, and $w'$ is an early shortcut if and only if $\doub^a_w(w')$ is an early shortcut. If $\glide^a_w(B)$ is defined, then $w\not\in\short(\glide^a_w(B))$.
\end{lemma}

\begin{proof}
The statements about doubling hold because the marked blocks $B$ and $\doub^a_w(B)$ differ only by a translation in $\Lambda^\ext$. The statement about gliding follows from the observation in Item~\ref{sh:shortcuts} that no vertex of $h\cdot \Lambda\subseteq B^\eps(h)$ is a shortcut.
\end{proof}

The following gives a commutation relation between certain doubles and glides, similar to Lemma~\ref{lem:doubling_commutes}.

\begin{lemma} \label{lem:gliding_commutes}
Let $B=B(\alpha,\eps,\kappa;h)$ be a marked block. Let $w,w'$ be boundary-adjacent vertices of $B$, and suppose that $w'\in\short B$. For any $a,b\ne0$, the marked blocks $A_1=\glide^b_{\doub^a_w(w')}\doub^a_w(B)$ and $A_2=\doub^a_{\glide^b_{w'}(w)}\glide^b_{w'}(B)$ are equal.
\end{lemma}

\begin{proof}
Since $w$ and $w'$ are adjacent, we have $\doub^a_w(w')=w'$ and $\glide_{w'}^b(w)=w$. By Lemma~\ref{lem:stay_shortcut}, we have that $w'$ is a (late) shortcut in $B$ if and only if it is a (late) shortcut in $\doub^a_w(B)$. Moreover, the depth of $w'$ inside $\doub^a_w(B)$ is equal to its depth inside $B$. Since doubling affects only the final parameter marked blocks, it follows from the expressions in Item~\ref{sh:gliding} that, in order to prove the lemma, we only need check that the final parameters of the blocks $A_1$ and $A_2$ agree. Let us refer to these parameters as $h_1$ and $h_2$, respectively.

Let $h\ell_0\cdot w_{j_0}$ be the shortcut of $B$ that is adjacent to $w'$ in $\Lambda^\ext$, with $\ell_0$ the minimal element of $\L(\eps,\kappa)$ making this true. Note that $w'=h\ell_0\cdot w_{j_1}$ for some $j_1$, and $\ell_0$ is also minimal for $w'$. Additionally, if $\ell\in\L(\eps,\kappa)$ is minimal such that $w=h\ell\cdot w_{j_2}$, then $\ell\in\{\ell_0,\ell_0t_{j_0}t_{j_1}\}$.

First consider the block $A_1$. From Item~\ref{sh:doubling}, we compute that $A_1=\glide^b_{w'}(B(\alpha,\eps,\kappa;h\ell t_{j_2}^a\ell^{-1}))$. Observe that if $\ell'\in\L(\eps,\kappa)$ is minimal such that $w'=h\ell t_{j_2}^a\ell^{-1}\ell'\cdot w_{j_1}$, then $\ell'=\ell_0$. From Item~\ref{sh:gliding}, we compute that 
\[
h_1 \,=\, h\ell t_{j_2}^a\ell^{-1}\ell_0t_{j_1}^{\hat b}.
\]
Now consider the block $A_2$. From Item~\ref{sh:gliding}, the final parameter of $\glide^b_{w'}(B)$ is equal to $h\ell_0t_{j_1}^{\hat b}$. This time, if $\ell'$ is minimal such that $w=h\ell_0t_{j_1}^{\hat b}\ell'\cdot w_{j_2}$, then $\ell'=\ell_0^{-1}\ell$. From Item~\ref{sh:doubling}, we compute
\[
h_2 \,=\, h\ell_0t_{j_1}^{\hat b}\ell't_{j_2}^a{\ell'}^{-1}.
\]

To complete the proof, we split into cases according to the value of $\ell$ and then compare the above expressions for $h_1$ and $h_2$. In the first case, $\ell=\ell_0$, and hence $\ell'=1$ in the expression for $h_2$. Since $t_{j_1}$ and $t_{j_2}$ commute, it is easy to see that $h_1=h_2$ in this case. The second case is that $\ell=\ell_0t_{j_0}t_{j_1}$, and hence $\ell'=t_{j_0}t_{j_1}$. This time a computation shows that
\[
h_1 \,=\, h_2 \,=\, h\ell_0t_{j_0}t_{j_1}^{\hat b}t_{j_2}^at_{j_0}^{-1}. \qedhere
\]
\end{proof}

We shall use the doubling and gliding operations to define a map $f\colon\Gamma^\ext\to\Lambda^\ext$, and the desired quasiisometric embedding $A_\Gamma\to A_\Lambda$ will arise in the process.

Thinking of $\Gamma^\ext$ as being obtained from $\bigsqcup_{g\in A_\Gamma}g\cdot \Gamma$, we aim to define $f$ inductively on the syllable length $\syl g$ of $g$. Recall that $\syl g$ is the minimal $k$ such that we can write $g=s_{i_1}^{n_1}\dots s_{i_k}^{n_k}$. 

\bsh{Reconstructing extension graphs} \label{sh:f_construction}
First define $f|_\Gamma=\iota_{+,1}$, so that $f$ identifies $\Gamma$ with (the boundary path) of the positive basic block $B^+(1)\subseteq\Lambda^\ext$. In other words, $f(\Gamma)=B(0,+,0;1)$. For readability, here and below we do not distinguish between a block and its boundary.

We now proceed inductively. Suppose that we have defined $f$ on all $g\cdot \Gamma$ with $\syl g\le d$, in such a way that $f|_{g\cdot \Gamma}$ is a marked block. Given $g\in A_\Gamma$ with $\syl g=d+1$, write $g=g's_i^a$, where $\syl{g'}=d$. The definition of $f|_{g\cdot \Gamma}$ depends on whether $f(g'\cdot v_i)$ is a shortcut of the marked block $B=f(g'\cdot \Gamma)$.
\[
\text{If } f(g'\cdot v_i)\not\in\short(B), \text{ then } f|_{g\cdot \Gamma} \,=\, \doub^a_{f(g'\cdot v_i)}(B).
\]
\[
\text{If } f(g'\cdot v_i)\in\short(B), \text{ then } f|_{g\cdot \Gamma} \,=\, \glide^a_{f(g'\cdot v_i)}(B).
\]
\esh

\emph{A priori}, the definition of $f|_{g\cdot \Gamma}$ depends on the choice of word representing $g$. Moreover, if $g\cdot v=gg'\cdot v$ represent the same vertex of $\Gamma^\ext$, then it is not immediately clear that $f|_{g\cdot \Gamma}(g\cdot v)=f|_{gg'\cdot \Gamma}(g\cdot v)$. The following lemma shows that these choices do not matter.

\begin{lemma} \label{lem:f_well_defined}
The map $f\colon\Gamma^\ext\to\Lambda^\ext$ is well defined.
\end{lemma}

\begin{proof}
First suppose that $g\cdot v_i=gs_{i'}\cdot v_i$, where $|gs_{i'}|>|g|$. For this to happen, it must be the case that $v_{i'}$ is equal to or adjacent to $v_i$ in $\Gamma$. Since the doubling and gliding operations both fix the $f$--images of the star of a vertex, we get that $f|_{g\cdot \Gamma}(g\cdot v_i)=f|_{gs_{i'}\cdot \Gamma}(g\cdot v_i)$, as desired.

Similarly, suppose that $g=g_1s_{i_1}^{a_1}=g_2s_{i_2}^{a_2}$, where both expressions are syllable-reduced words. For this to be the case, $s_{i_1}$ and $s_{i_2}$ must commute, and we can write $g=g's_{i_1}^{a_1}s_{i_2}^{a_2}$. The result therefore follows from Lemmas~\ref{lem:doubling_commutes} and~\ref{lem:gliding_commutes}, because $f(g'\cdot v_{i_1})$ and $f(g'\cdot v_{i_2})$ are boundary-adjacent vertices of the block $f(g'\cdot \Gamma)$ and so cannot both be shortcuts.
\end{proof}

In view of Lemma~\ref{lem:f_well_defined}, there is no ambiguity in writing $f(\gamma\cdot v)$ for a vertex $v\in\Gamma$ as $\gamma$ varies among syllable-reduced representatives of a fixed element of $A_\Gamma$, so we shall sometimes do this. 

\bsh{Notational update}
Now that we have defined the map $f$ using doubling and gliding, we shall not need to make further direct reference to the cyclic ordering on the vertices of $C_m$ and $C_n$ that we established above. To reduce the number of subscripts needed in what follows, we therefore drop this identification. From now on, $v$ will denote an arbitrary vertex of $C_m=\Gamma$, with corresponding generator $s\in S$, and $w_1,w_2$ will denote an arbitrary pair of (possibly equal) vertices of $C_n=\Lambda$, corresponding to $t_1,t_2\in T$.
\esh

From the map $f$ we can identify a ``route'' map, that keeps track of how the labels of the vertices of the blocks $f(g\cdot \Gamma)$ evolve as we increase the syllable length of $g$. 

\bsh{The lifted map $F$} \label{sh:roots}
We define a map $F$ of syllable-reduced words. We do this recursively on syllable length in such a way that $F(\gamma)$ represents an element of $A_\Lambda$ satisfying the property that $F(\gamma)\cdot \Lambda$ is one of the copies of $\Lambda$ appearing in the block $f(\gamma\cdot \Gamma)$. Intuitively speaking, $F(\gamma)$ is the label of the copy of $\Lambda$ in $f(\gamma\cdot \Gamma)$ that is closest to the base copy $1\cdot\Lambda\subseteq\Lambda^\ext$.

Set $F(1)=1$, and indeed $1\cdot \Lambda$ appears in $f(1\cdot \Gamma)$. Let $\gamma=s_1^{a_1}\dots s_d^{a_d}$ be a syllable-reduced word in $A_\Gamma$.
Assume we have defined $F(\gamma')$ for $\gamma'=s_1^{a_1}\dots s_{d-1}^{a_{d-1}}$, such that $F(\gamma')\cdot \Lambda$ is one of the copies of $\Lambda$ appearing in the block $f(\gamma'\cdot \Gamma)$. 
That means that there exists a minimal basic relative label $\ell\in \cal{L}$ and some $w\in\Lambda$ such that $f(\gamma'\cdot v_d)=F(\gamma')\ell\cdot w$. 


Let $t$ be the standard generator of $A_\Lambda$ corresponding to $w$. If $f(\gamma'\cdot v_d)$ is not a shortcut of $B$, then we set
\[
F(\gamma) \,=\, F(\gamma')\ell t^{a_d}.
\]
If $f(\gamma'\cdot v_d)\in\short(B)$, then $v_d$ is a boundary vertex of $B$, so the expression for $F(\gamma)$ depends on the position of $F(\gamma')\cdot \Lambda$ inside $B$. If no vertex of $F(\gamma')\cdot \Lambda$ has greater depth inside $B$ than $f(\gamma'\cdot v_d)$, then we define
\[
F(\gamma) \,=\, F(\gamma')\ell t^{\hat a_d},
\]
where $\hat a_d$ is as defined in Item~\ref{sh:gliding}. Otherwise, $f(\gamma'\cdot v_d)\in\short(B)$ and there are vertices of $F(\gamma')\cdot \Lambda$ of greater depth inside $B$ than $f(\gamma'\cdot v_d)$, and in this case we declare
\[
F(\gamma) \,=\, F(\gamma')\ell t^{\hat a_d-1}.
\]
Observe that the definitions of doubling and gliding imply that $F(\gamma)\cdot \Lambda$ is one of the copies of $\Lambda$ appearing in $f(\gamma\cdot \Gamma)$.

In summary, if we denote by $\gamma_c$ the subword of $\gamma$ consisting of the first $c$ syllables, $s_1^{a_1}\dots s_c^{a_c}$, then $F(\gamma_c)$ is a subword of $F(\gamma)$. We have that $F(\gamma)=\ell_1t_1^{b_1}\dots \ell_dt_d^{b_d}$ where $\ell_c$ is minimal such that $f(\gamma_{c-1}\cdot v_c)= F(\gamma_{c-1})\ell_c\cdot w$ for some $w\in \Lambda$ and $b_c\in \set{a_c,\hat{a}_c,\hat{a}_c-1}\subseteq\set{a_c-1,a_c,a_c+1}$.
For each $c\leq d$, if $\ell_c\neq 1$, then because it is minimal it must have a final letter that does not commute with $t_c$ and in particular $\ell_c$ does not commute with $t_c$.
\esh

The map $F$ can be thought of as a lift of $f$, in the sense that, at least on the level of vertex sets, it induces $f$. Indeed, if $P=\{gs^k\,:\,k\in\Z\}$ is a standard geodesic of $A_\Gamma$ corresponding to a vertex $g\cdot v\in\Gamma^\ext$, then $F$ maps $P$ to a standard geodesic that corresponds to $f(g\cdot v)\in\Lambda^\ext$. This will not be needed in our arguments.

\begin{lemma} \label{lem:F_well_defined}
The map $F$ induces a well-defined map $A_\Gamma\to A_\Lambda$.
\end{lemma}

\begin{proof}
We must show that $F(g)$ is independent of the syllable-reduced word representing $g\in A_\Gamma$. We do this by induction on $\syl g$. The result is clear when $\syl g\le1$. 

Suppose that we have proved the lemma for elements of syllable length less than $d$, and let $g\in A_\Gamma$ have $\syl g=d\ge2$. Represent $g$ by a syllable-reduced word $\gamma=\gamma_0s_{1}^{a_1}s_{2}^{a_2}$. If $s_{1}$ and $s_{2}$ do not commute then there is nothing to prove, so suppose that $s_{1}$ and $s_{2}$ commute. By the inductive assumption, it suffices to show that $F(\gamma_0s_1^{a_1}s_2^{a_2})=F(\gamma_0s_2^{a_2}s_1^{a_1})$.

If $\gamma_0\neq 1$, let $\omega_0$ be the vertex along which we either doubled or glided to produce the block $B_0=f(\gamma_0\cdot \Gamma)$. Let $\omega_1=f(\gamma_0\cdot v_1)$, which is the vertex of the block $B_0$ along which we either double or glide to produce the block $B_1=f(\gamma_0s_1^{a_1}\cdot \Gamma)$.  Similarly, let $\omega_2=f(\gamma_0\cdot v_2)$ and $B_2=f(\gamma_0s_2^{a_2}\cdot\Gamma)$. Note that $\omega_1$ and $\omega_2$ are boundary-adjacent in $B_0$, because $s_{1}$ and $s_{2}$ commute. After relabelling, we can therefore assume that $\omega_1$ does not lie at the intersection of two of the copies of $\Lambda$ that make up $B_0$.

Let $\ell_{0,1}\in\L$ be minimal such that $\omega_1=F(\gamma_0)\ell_{0,1}\cdot w_1$, and define $\ell_{0,2}$ similarly. Since $\omega_1$ and $\omega_2$ are adjacent, $t_{1}$ and $t_{2}$ commute. In order to compute the two expressions for $F(g)$, we also need to consider relative labels for the blocks $B_1$ and $B_2$. Let $\ell_{1,2}\in\L$ be minimal such that $\omega_2=F(\gamma_0s_1^{a_1})\ell_{1,2}\cdot w_{2}$, and let $\ell_{2,1}\in\L$ be minimal such that $\omega_1=F(\gamma_0s_2^{a_2})\ell_{2,1}\cdot w_{1}$.

We split into cases according to whether $\omega_2$ is a shortcut of $B_0$.

\medskip\noindent\textbf{\linkdest{wd_dd}{Case 1}. $\omega_2\not\in\short(B_0)$. }
If $\omega_2$ is not a shortcut of $B_0$, then $B_1$ and $B_2$ are both obtained from $B_0$ by doubling, and the expression of Item~\ref{sh:roots} gives
\begin{align}
F(\gamma_0s_1^{a_1}s_2^{a_2}) \,=\, F(\gamma_0)\ell_{0,1}t_{1}^{a_1}\ell_{1,2}t_{2}^{a_2}, \quad
F(\gamma_0s_2^{a_2}s_1^{a_1}) \,=\, F(\gamma_0)\ell_{0,2}t_{2}^{a_2}\ell_{2,1}t_{1}^{a_1}. \tag{$*$} \label{eq:wd_dd}
\end{align}

If $\ell_{0,1}=\ell_{0,2}$ then it follows that $\ell_{1,2}=\ell_{2,1}=1$. Since $t_{1}$ and $t_{2}$ commute, the expressions in~\eqref{eq:wd_dd} define the same element of $A_\Lambda$.

If $\ell_{0,1}\ne\ell_{0,2}$, then $\omega_2$ lies at the intersection of two copies of $\Lambda$ that make up $B_0$ and $|\ell_{0,2}|<|\ell_{0,1}|$. Since $\omega_2\not\in\short(B_0)$, it is adjacent to a non-boundary vertex $w\in B_0$. Let $t=\type(w)$, which commutes with $t_{2}$.

If $\gamma_0=1$ or $\dep(\omega_0;B_0)\le\dep(\omega_2;B_0)$, then $\ell_{0,1}=\ell_{0,2}t$. It can also easily be seen that $\ell_{1,2}=1$ and $\ell_{2,1}=t$. Since $t_{2}$ commutes with both $t$ and $t_{1}$, the expressions in~\eqref{eq:wd_dd} define the same element of $A_\Lambda$. 

If $\dep(\omega_0;B_0)>\dep(\omega_2;B_0)$, then $\ell_{0,1}=\ell_{0,2}t^{-1}$, and we have $\ell_{1,2}=1$ and $\ell_{2,1}=t^{-1}$. Again, the expressions in~\eqref{eq:wd_dd} define the same element of $A_\Lambda$.

\medskip\noindent\textbf{\linkdest{wd_dg}{Case 2}. $\omega_2\in\short(B_0)$. }
If $\omega_2$ is a shortcut of $B_0$, then $B_2$ is obtained from $B_0$ by gliding. Also, by Lemma~\ref{lem:stay_shortcut}, the block $f(g\cdot \Gamma)$ is obtained from $B_1$ by gliding, whereas it is obtained from $B_2$ by doubling. We do not have uniform expressions for $F(\gamma_0s_1^{a_1})$ and $F(\gamma_0s_2^{a_2})$ this time, because the expressions in Item~\ref{sh:roots} depend on depth considerations. Let $w\in B_0$ be the shortcut twin of $\omega_2$, and let $t=\type(w)$, which commutes with $t_{2}$. 

First suppose that $\ell_{0,1}=\ell_{0,2}=\ell$. As in Case~\hyperlink{wd_dd}{1}, this implies that $\ell_{1,2}=\ell_{2,1}=1$. If $\dep(\omega_0;B_0)\le\dep(\omega_2;B_0)$, then from the expressions in Item~\ref{sh:roots} we compute that
\[
F(\gamma_0s_1^{a_1}s_2^{a_2}) \,=\, F(\gamma_0)\ell t_{1}^{a_1}t_{2}^{\hat a_2}, \quad
F(\gamma_0s_2^{a_2}s_1^{a_1}) \,=\, F(\gamma_0)\ell t_{2}^{\hat a_2}t_{1}^{a_1}.
\]
These expressions define the same element of $A_\Lambda$ because $t_{1}$ and $t_{2}$ commute. If, on the other hand, $\dep(\omega_0;B_0)>\dep(\omega_2;B_0)$, then the expressions in Item~\ref{sh:roots} give
\[
F(\gamma_0s_1^{a_1}s_2^{a_2}) \,=\, F(\gamma_0)\ell t_{1}^{a_1}t_{2}^{\hat a_2-1}, \quad
F(\gamma_0s_2^{a_2}s_1^{a_1}) \,=\, F(\gamma_0)\ell t_{2}^{\hat a_2-1}t_{1}^{a_1},
\]
which again define the same element of $A_\Lambda$.

Now suppose instead that $\ell_{0,1}\ne\ell_{0,2}$. As in Case~\hyperlink{wd_dd}{1}, this implies that $|\ell_{0,2}|<|\ell_{0,1}|$. If $\dep(\omega_0;B_0)<\dep(\omega_2;B_0)$, then $\ell_{0,1}=\ell_{0,2}t_{2}t$. It can also be seen that $\ell_{1,2}=1$ and $\ell_{2,1}=t$. 
Moreover, $\omega_1$ is a vertex of $F(\gamma_0s_1^{a_1})\Lambda$ that has greater depth inside $B_1$ than $\omega_2$. From the expressions in Item~\ref{sh:roots}, we compute
\[
F(\gamma_0s_1^{a_1}s_2^{a_2}) \,=\, F(\gamma_0)\ell_{0,2}t_{2}tt_{1}^{a_1}t_{2}^{\hat a_2-1}, \quad
F(\gamma_0s_2^{a_2}s_1^{a_1}) \,=\, F(\gamma_0)\ell_{0,2}t_{2}^{\hat a_2}tt_{1}^{a_1}.
\]
Since $t_{2}$ commutes with both $t$ and $t_{1}$, these expressions define the same element of $A_\Lambda$.

Finally, suppose that $\ell_{0,1}\ne\ell_{0,2}$ and $\dep(\omega_0;B_0)>\dep(\omega_2;B_0)$. In this case, $\ell_{0,1}=\ell_{0,2}t^{-1}t_{2}^{-1}$. This time we have $\ell_{1,2}=1$, and $\ell_{2,1}=t^{-1}$, because $w$ is a non-boundary vertex of the block $B_2$, as $B_2$ is obtained from $B_0$ by gliding. 
There are vertices of $F(\gamma_0)\cdot \Lambda$ with greater depth than $\omega_2$ inside $B_0$, but no vertex of $F(\gamma_0s_1^{a_1})\cdot \Lambda$ has greater depth inside $B_1$ than $\omega_2$. The expressions of Item~\ref{sh:roots} therefore yield
\[
F(\gamma_0s_1^{a_1}s_2^{a_2}) \,=\, F(\gamma_0)\ell_{0,2}t^{-1}t_{2}^{-1}t_1^{a_1}t_{2}^{\hat a_2}, \quad
F(\gamma_0s_2^{a_2}s_1^{a_1}) \,=\, F(\gamma_0)\ell_{0,2}t_{2}^{\hat a_2-1}t^{-1}t_{1}^{a_1},
\]
which again define the same element of $A_\Lambda$. This completes the proof.
\end{proof}

In view of Lemma~\ref{lem:F_well_defined}, we also write $F\colon A_\Gamma\to A_\Lambda$. 

Our goal will be to show that $F$ is a quasiisometric embedding. Each syllable-reduced word representing $g$ gives a word representing $F(g)$, and we can use an arbitrary such word to show that $F$ is Lipschitz. Showing that $F$ is coarsely colipschitz is more challenging, and for that we will need to understand how the images of different words representing $g$ differ.

\bsh{Changing representative of $F(g)$} \label{sh:rep_F_change}
Assume $g\in A_\Gamma$ is represented by a syllable-reduced word $\gamma=s_1^{a_1}\dots s_d^{a_d}$. Every other syllable-reduced word representing $g$ can be obtained by interchanging commuting syllables.
Let $\gamma'$ be such a word. That is,
\[\gamma'=s_1^{a_1}\dots s_{c-1}^{a_{c-1}}s_{c+1}^{a_{c+1}}\dots s_{c+k}^{a_{c+k}}s_c^{a_c}s_{c+k+1}^{a_{c+k+1}}\dots s_d^{a_d}\] 
is the syllable-reduced word representing $g$ obtained from shuffling $s_c^{a_c}$ past $k$ syllables that commute with it.
We will investigate the two words $\lambda=F(\gamma)$ and $\lambda'=F(\gamma')$ representing $F(g)$.

We write $\lambda=\ell_1t_1^{b_1}\dots \ell_dt_d^{b_d}$ and $\lambda'=\ell'_1{t'}_1^{b'_1}\dots \ell'_d{t'}_d^{b'_d}$, as in Item~\ref{sh:roots}.

Let $\gamma_e$ and $\gamma'_e$ be the subwords consisting of the first $e$ syllables of $\gamma$ and $\gamma'$ respectively.
By the definition of $F$ in Item~\ref{sh:roots} we have that $\lambda_e=F(\gamma_e)$ and $\lambda_e'=F(\gamma_e')$ are subwords of $\lambda$ and $\lambda'$ respectively.

\begin{lemma}
    We have
    \begin{itemize}
        \item $\lambda_{c-1}=\lambda'_{c-1}$,
        \item $t_c$ commutes with $\ell_{c+1}t_{c+1}^{b_{c+1}}\dots \ell_{c+k}t_{c+k}^{b_{c+k}}$,
        \item If $\ell_{c+1}=1$ then $\ell'_c=\ell_c$, otherwise $\ell'_c=\ell_c\ell_{c+1}$ as elements of $A_{\Lambda}$, and 
        \item For $e=c+\kappa$ where $1\leq \kappa<k$ we have $\ell'_{c+\kappa}=\ell_{c+\kappa+1}$ and $\ell'_{c+k}=1$.
    \end{itemize}
\end{lemma}

\begin{proof}
Note that $\gamma'_{c-1}=\gamma_{c-1}$, and so $\lambda'_{c-1}=\lambda_{c-1}$. To ease notation, let us write $\omega_c=f(\gamma_{c-1}\cdot v_{c})$ and $\omega_{c+1}=f(\gamma_{c-1}\cdot v_{{c+1}})$. By commutation, the vertices $\omega_c$ and $\omega_{c+1}$ are adjacent in the block $f(\gamma_{c-1}\cdot \Gamma)$. Hence $t_{c}$ commutes with $t_{{c+1}}$. If $\ell_{c+1}\ne1$, then it must be that $\omega_c$ lies at the intersection of two of the copies of $\Lambda$ that make up the block $f(\gamma_{c-1}\cdot \Gamma)$, and $\ell_{c+1}$ commutes with $t_{c}$. Thus $t_{c}$ commutes with $\ell_{c+1}t_{{c+1}}$. Iterating this same argument, we find that $t_{c}$ commutes with the entire word $\ell_{c+1}t_{{c+1}}^{b_{c+1}}\dots\ell_{c+k}t_{{c+k}}^{b_{c+k}}$.



As noted above, $\omega_c$ and $\omega_{c+1}$ are adjacent. If $\ell_{c+1}=1$, then $\ell'_c=\ell_c$. Otherwise, the minimal element of $\L$ required to pass from $F(\gamma_{c-1})\cdot \Lambda$ to a copy of $\Lambda$ containing $\omega_{c+1}$ differs from $\ell_c$, and the difference is $\ell_{c+1}$. But that minimal element is precisely $\ell'_c$. Thus $\ell'_c=\ell_c\ell_{c+1}$ as elements of $A_\Lambda$, though possibly these are different representatives of that element.

If $k>1$, then since $s_{{c+2}}$ commutes with $s_{c}$ and $\gamma$ is syllable-reduced, the two neighbours of the vertex $v_{c}$ in $\Gamma=C_m$ are $v_{{c+1}}$ and $v_{{c+2}}$. Thus $v_{c+1}=v_{c+\kappa}$ for all odd $\kappa\le k$, and $v_{c+2}=v_{c+\kappa}$ for all even $\kappa\le k$. Since $\ell_{c+\kappa}$ does not commute with $t_{{c+\kappa}}$, for each positive $\kappa<k$ we have $\ell'_{c+\kappa}=\ell_{c+\kappa+1}$. These are all trivial if $\omega_c$ does not lie at the intersection of two copies of $\Lambda$ that make up $f(\gamma_{c-1}\cdot \Gamma)$, and otherwise $\ell'_{c+\kappa+1}=(\ell'_{c+\kappa})^{-1}$ commutes with $t_{c}$. Finally, we have $\ell'_{c+k}=1$, because $\omega_c$ is adjacent to $f(\gamma_{c-1}\cdot v_{i_{c+k}})$. 
\end{proof}
%
%
\esh


We wish to find short representatives of $F(g)$. We do this by making a certain choice of an order in which to read the letters of $g$, obtaining a representative of $g$ that we call \emph{lazy}, and then applying the definition of $F$ to this lazy representative. Recall that $t$ denotes the standard generator of $A_\Lambda$ corresponding to the vertex $w$ of $\Lambda$.

\bsh{Lazy representatives in $A_\Gamma$} \label{sh:lazy}
Intuitively, a syllable-reduced word in $A_\Gamma$ is lazy if whenever there is a choice of two vertices to extend along in the construction of $f$, it chooses the one that is closer to the previous vertex we extended along.
See Figure~\ref{fig:lazy}.

We note that if $s_1^{a_1}\dots s_{d}^{a_d}\in A_\Gamma$ is such that $s_{c+1}$ and $s_{c+2}$ commute for some $c\leq d-2$, then $f(s_1^{a_1}\dots s_{c}^{a_c}\cdot v_{c+1})$ and $f(s_1^{a_1}\dots s_{c}^{a_c}\cdot v_{c+2})$ are boundary-adjacent in the block $f(s_1^{a_1}\dots s_c^{a_c}\cdot\Gamma)$.

\begin{definition}[Lazy]
Let $\gamma=s_{1}^{a_1}\dots s_{d}^{a_d}$ be a syllable-reduced word. We say that $\gamma$ is \emph{lazy} if it satisfies the following condition for every $c\leq d-2$ for which $s_{{c+1}}$ and $s_{{c+2}}$ commute. 
    Let $\gamma_c=s_{1}^{a_1}\dots s_{c}^{a_c}$, and 
    for $k\in \set{1,2}$ let $\ell_k\in \cal{L}$ be the minimal basic relative labels satisfying $f(\gamma_c\cdot v_{{c+k}})=F(\gamma_c)\ell_k\cdot w_k$. We require the following.
    \begin{itemize}
        \item If $\ell_1\neq \ell_2$, then $|\ell_1|<|\ell_2|$.
        \item If $\ell_1=\ell_2\neq 1$, then $t_2$ does not commute with any final letter of $\ell_1$. (At most one of the $t_{k}$ can commute with a final letter of $\ell_1$.)
        \item If $\ell_1=\ell_2=1$, then $t_{2}$ does not. (At most one of the $t_{k}$ can commute with $\type(f(\gamma_c\cdot v_c))$.) 
    \end{itemize}
\end{definition}


\esh

\begin{figure}[ht]
\centering
\includegraphics[width=15cm,trim= 0 10mm 0 15mm]{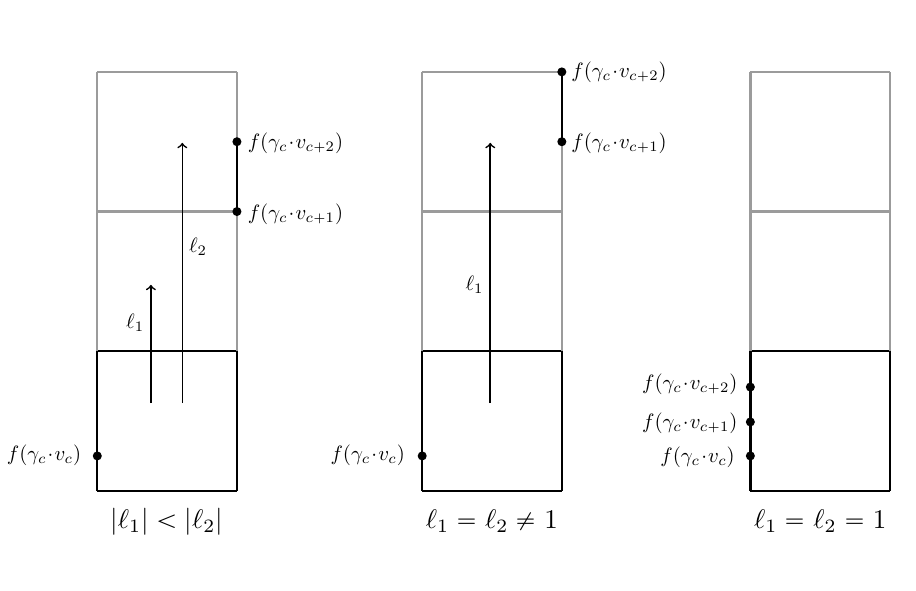}
\caption{The conditions defining laziness of a syllable-reduced word over $S$.} \label{fig:lazy}
\end{figure}

\begin{lemma} \label{lem:lazy_exists}
Every $g\in A_\Gamma$ has at least one lazy representative.
\end{lemma}

\begin{proof}
The lemma is vacuously true for elements of syllable length at most one. For $d\ge2$, suppose that we have proved the lemma for elements of syllable length less than $d$. Given $g\in A_\Gamma$ with $\syl g=d$, let $\gamma=s_{1}^{a_1}\dots s_{d}^{a_d}$ be a syllable-reduced word representing $g$. For $c<d$, let $\gamma_c=s_{1}^{a_1}\dots s_{c}^{a_c}$.
By induction, we can assume that $\gamma_{d-1}$ is lazy. 

If the generators $s_{d}$ and $s_{{d-1}}$ do not commute then $\gamma$ is lazy. Otherwise, $f(\gamma_{d-2}\cdot v_{{d-1}})$ and $f(\gamma_{d-2}\cdot v_{d})$ are boundary-adjacent vertices of the block $f(\gamma_{d-2}\cdot \Gamma)$. For $k\in\{0,1\}$, let $\ell_k\in\L$ be minimal such that $f(\gamma_{d-2}\cdot v_{{d-k}}) = F(\gamma_{d-2})\ell_k\cdot w_{k}$. If any of the conditions of Item~\ref{sh:lazy} are met, then $\gamma$ is again lazy. Otherwise, since $s_{d}$ and $s_{{d-1}}$ commute, we can also represent $g$ by the syllable-reduced word $\gamma_{d-2}s_{d}^{a_d}s_{{d-1}}^{a_{d-1}}$. 


By induction, there is a lazy representative $\gamma'$ of the element represented by $\gamma_{d-2}s_{d}^{a_d}$, which is obtained by shuffling syllables. Let $s_{c}^{a_c}$ be the final syllable of $\gamma'$. We claim that $\gamma's_{{d-1}}^{a_{d-1}}$ is lazy. If not, then it is because $s_{c}$ commutes with $s_{{d-1}}$ and does not satisfy the conditions in Item~\ref{sh:lazy}. If $s_{c}=s_{d}$ then the conditions are satisfied by assumption, so $s_{c}\ne s_{d}$. But in order for $s_{c}$ to be the final letter of $\gamma'$, it must commute with $s_{d}$ and be distinct from both $s_{d}$ and $s_{{d-1}}$, because $\gamma$ is syllable-reduced. But now the generators $s_{c}$, $s_{d}$, and $s_{{d-1}}$ commute pairwise, which is impossible.
\end{proof}

The following is a key step in showing that the map $F\colon A_\Gamma\to A_\Lambda$ is a quasiisometric embedding. Recall that $T$ denotes the set of generators of $A_\Lambda$ and their inverses. 

\begin{proposition} \label{prop:roots_structure}
Let $\gamma=s_{1}^{a_1}\dots s_{d}^{a_d}$ be a lazy word. Let $\gamma_c$ be the initial substring ending with $s_{c}^{a_c}$, and let $\lambda_c=F(\gamma_c)$. The word $\lambda=\lambda_d=\ell_1t_{1}^{b_1}\dots\ell_dt_{d}^{b_d}$ can only fail to be syllable-reduced if the following happens for some values of $c$ and $k\ge0$:
\begin{itemize}
\item   the label $\ell_{c+k+1}\ne1$ begins with letters $\tau_1^{\epsilon_{c+k+1}}$ and $\tau_2^{\epsilon_{c+k+1}}$, with $\epsilon_{c+k+1}\in\{1,-1\}$;
\item   the label $\ell_c\ne1$ ends with two letters $\tau_1^{\epsilon_c}$ and $\tau_2^{\epsilon_c}$, such that $\tau_1,\tau_2\in T$ correspond to a shortcut pair in $f(g_{c-1}\cdot \Gamma)$, with $\epsilon_c\in\{1,-1\}$;
\item   after relabelling $\tau_1$ and $\tau_2$, the generators $t_{c},\dots,t_{{c+k}}$ all commute with $\tau_1$ but not $\tau_2$;
\item   if $k>0$, then $\ell_{c+1}=\dots=\ell_{c+k}=1$.
\end{itemize}
Let $\lambda'$ be the word obtained from $\lambda$ as follows. For each pair $(c,k)$ as above, shuffle the last letter of $\ell_c$ past $t_{c},\dots,t_{{c+k}}$, then cancel $\tau_1^{\epsilon_c}$ with $\tau_1^{\epsilon_{c+k+1}}$ if $\epsilon_c=-\epsilon_{c+k+1}$. The word $\lambda'$ is syllable-reduced, and in particular has minimal word-length among representatives of $F(g)$.
\end{proposition}

\begin{proof}
There is nothing to prove when $d=0$. Suppose we have proven the proposition for all $c<d$. 
By the inductive hypothesis, we only need to check for failures of $\lambda$ to be syllable-reduced that arise from appending $\ell_dt_{d}^{b_d}$ to $\lambda_{d-1}$. We do this by splitting into a number of cases. 


In short, we shall prove that if the addition of $\ell_dt_{d}^{b_d}$ to the end of $\lambda_{d-1}$ causes a new failure of $\lambda$ to be syllable-reduced, then it is because there is some $c<d$ such that $\beta_{c+2}$ is obtained from $\beta_{c+1}$ by gliding (or would be, if $c=d-1$), and $\beta_{c+r+1}$ is obtained from $\beta_{c+r}$ by doubling for all $r\ge2$. Note that whilst we show that this is necessary to cause an obstruction, it is not sufficient. If it does cause an obstruction, then the issue is with a final letter of $\ell_c$ and an initial letter of $\ell_d$, which in particular requires $\ell_d\ne1$.

In Case~\hyperlink{structure_case_1}{1} we will show that if $\ell_d=1$, then all obstructions to $\lambda$ being syllable-reduced come from $\lambda_{d-1}$. In Case~\hyperlink{structure_case_2}{2} we will assume that $\ell_d\neq 1$ and show that if there is an obstruction to $\lambda$ being syllable-reduced coming from $\ell_dt_{d}^{b_d}$, then it is caused by a first letter of $\ell_d$. We then further divide into cases by considering whether the vertex corresponding to that letter is a shortcut. We inductively reduce to the case where it is, and in that case show that any obstruction it causes has the form stated in the proposition.

To ease notation, for $c\le d$ let us write $\omega_c$ for the vertex $f(\gamma_{c-1}\cdot v_{c})=f(\gamma_c\cdot v_{c})$ of $\Lambda^\ext$. We also write $\beta_c$ for the block $f(\gamma_{c-1}\cdot \Gamma)$, so that $\omega_{c-1}$ and $\omega_c$ both lie in $\beta_c$.

\medskip\noindent\textbf{\linkdest{structure_case_1}{Case 1}. $\ell_d=1$. }
If $\ell_d=1$, then any failure of $\lambda$ to be syllable-reduced that arises from $\ell_dt_{d}^{b_d}$ requires that $t_{d}$ commutes with $t_{{d-1}}$. Since $\ell_d=1$, the vertices $\omega_d$ and $\omega_{d-1}$ lie in a common copy of $\Lambda$ inside the block $\beta_d$, so we must have $t_{d}\ne t_{{d-1}}$. Thus $\omega_d$ and $\omega_{d-1}$ are adjacent vertices of a copy of $\Lambda$ that makes up the block $\beta_{d-1}$. Note that this was independent of whether $\beta_d$ was obtained from $\beta_{d-1}$ by doubling or gliding.

If $\ell_{d-1}=1$ and $t_{d}^{b_d}$ obstructs $\lambda$ being syllable-reduced, then as well as commuting with $t_{{d-1}}$, it must be that $t_{d}$ commutes with $t_{{d-2}}$. Since $\ell_{d-1}=\ell_d=1$, the vertices $\omega_d$ and $\omega_{d-2}$ live in a common copy of $\Lambda$ inside $\beta_{d-1}$, so they cannot be equal as the word representing $g$ was syllable-reduced. If they are adjacent, then since $\Lambda$ has no triangles, we have that $\omega_d$ is adjacent to both $\omega_{d-1}$ and $\omega_{d-2}$ but the latter two are not adjacent, which is impossible because the word representing $g$ was lazy. 

Similarly, if $\ell_{d-1}\ne1$ and $t_{d}^{b_d}$ obstructs $\lambda$ being syllable-reduced, then $t_{d}$ must commute with a final letter of $\ell_{d-1}$, which is impossible because we started with a lazy representative of $g$. 

We have shown that if $\ell_d=1$ then all failures of $\lambda$ to be syllable-reduced come from $\lambda_{d-1}$.

\medskip\noindent\textbf{\linkdest{structure_case_2}{Case 2}. $\ell_d\ne1$. }
Now suppose that $\ell_d\ne1$. By the definition of $\ell_d$, it has a final letter that does not commute with $t_{d}$, and no final letter is equal to $t_{d}^{\pm1}$. Thus if $\ell_dt_{d}^{b_d}$ obstructs $\lambda$ being syllable-reduced, then this failure is caused by the first letter of $\ell_d$.

As shown by Lemma~\ref{lem:stay_shortcut}, the vertex $\omega_{d-1}$ is a non-shortcut boundary vertex of the block $\beta_d$. In particular, $t_{{d-1}}^{\pm1}$ cannot be the first letter of $\ell_d$. Thus if there is an obstruction to $\lambda$ being syllable-reduced from $\ell_d$, then exactly one possible first letter of $\ell_d$ must commute with $t_{{d-1}}$. Let us write $\tau$ for this letter, and $\omega_\tau$ for the vertex of $\beta_d$ that it corresponds to. Note that $\omega_{d-1}$ is adjacent to $\omega_\tau$, because $\tau$ and $t_{{d-1}}$ commute. 

We shall split into three cases, according to whether $\omega_\tau$ is a shortcut of $\beta_d$ and whether $\omega_{d-1}$ is a shortcut of $\beta_{d-1}$. We start with the assumption that $\omega_{d-1}$ is adjacent to $\omega_\tau$.

\medskip\noindent\textbf{\linkdest{structure_case_2a}{Case 2.a}. $\omega_\tau\in\short\beta_d$. }
Suppose that $\omega_\tau\in\short\beta_d$. In this case, the first two letters of $\ell_d$ are $\tau$ and the type of the shortcut twin $\omega'$ of $\omega_\tau$. Note that $\omega'$ is not adjacent to $\omega_{d-1}$, so its type does not commute with $t_{{d-1}}$. Since $\omega_{d-1}$ is adjacent to $\omega_\tau$ and is not a shortcut of $\beta_d$ by Lemma~\ref{lem:stay_shortcut}, it is a non-shortcut boundary vertex of $\beta_{d-1}$ adjacent to $\omega_\tau$.

If $\ell_{d-1}=1$, then $\omega_{d-2}$ lies in a common copy of $\Lambda$ with $\omega_{d-1}$ and $\omega_\tau$. By Lemma~\ref{lem:stay_shortcut}, it is not a shortcut of $\beta_{d-1}$. Moreover, it is not equal to $\omega_{d-1}$ because we started with a syllable-reduced representative of $g$. This shows that $\omega_{d-2}$ is not equal or adjacent to $\omega_\tau$, so $t_{{d-2}}$ does not commute with $\tau$. Hence $\tau$ does not obstruct $\lambda$ being syllable-reduced.

If $\ell_{d-1}\ne1$, then either $\tau^{\pm1}$ is a final letter of it or it is not. If it is not, then it does not commute with any final letter of $\ell_{d-1}$, and therefore cannot obstruct $\lambda$ being syllable-reduced. If it is a final letter of $\ell_{d-1}$, then $\tau$ does obstruct $\lambda$ being syllable-reduced, and it does so exactly as described in the proposition, with $c=d-1$, $k=0$.


\medskip\noindent\textbf{\linkdest{structure_case_2b}{Case 2.b}. $\omega_\tau\not\in\short\beta_d$ and $\omega_{d-1}\in\short\beta_{d-1}$. }
Suppose instead that $\omega_\tau\not\in\short\beta_d$, and assume that $\omega_{d-1}\in\short\beta_{d-1}$. This means that $\beta_d$ was obtained from $\beta_{d-1}$ by gliding along $\omega_{d-1}$, and $\omega_\tau$ is the shortcut twin of $\omega_{d-1}$ in $\beta_{d-1}$. Consequently, if $\ell_{d-1}\ne1$, then its final letters cannot commute with $\tau$, so $\tau$ does not obstruct $\lambda$ being syllable-reduced.

If $\ell_{d-1}=1$, then $\omega_{d-2}$ lies in a common copy of $\Lambda$ with both $\omega_{d-1}$ and $\omega_\tau$. By Lemma~\ref{lem:stay_shortcut}, it is not a shortcut of $\beta_{d-1}$, so it is not equal to $\omega_\tau$. Thus if $\tau$ obstructs $\lambda$ being syllable-reduced, then $\omega_{d-2}$ must be adjacent to $\omega_\tau$ in $\beta_{d-1}$, and $\omega_\tau\in\short\beta_{d-1}$. We are reduced to the situation of Case~\hyperlink{structure_case_2a}{2.a}, but with $d$ replaced by $d-1$.

\medskip\noindent\textbf{\linkdest{structure_case_2c}{Case 2.c}. $\omega_\tau\notin\short\beta_d$ and $\omega_{d-1}\not\in\short\beta_{d-1}$. }
Finally, suppose that $\omega_\tau\not\in\short\beta_d$ and also $\omega_{d-1}\not\in\short\beta_{d-1}$. In this case, $\beta_d$ was obtained from $\beta_{d-1}$ by doubling along $\omega_{d-1}$, so $\omega_\tau$ is a non-boundary vertex of $\beta_{d-1}$. Hence $\omega_{d-1}$ lies at the intersection of two copies of $\Lambda$ that make up $\beta_{d-1}$, because it is adjacent to $\omega_\tau$.

If $\ell_{d-1}\ne1$, then its final letter cannot commute with $\tau$, so $\tau$ cannot obstruct $\lambda$ being syllable-reduced.

If $\ell_{d-1}=1$, then again $\omega_{d-2}$ lies in a common copy of $\Lambda$ with $\omega_{d-1}$ and $\omega_\tau$. Thus if $\tau$ obstructs $\lambda$ being syllable-reduced, then $\omega_{d-2}$ must be adjacent to $\omega_\tau$, which is not a shortcut of $\beta_{d-1}$. By considering whether or not $\omega_{d-2}$ is a shortcut of $\beta_{d-2}$, we are reduced to either case~\hyperlink{structure_case_2b}{2.b} or Case~\hyperlink{structure_case_2c}{2.c}, but with $d$ replaced by $d-1$.

\medskip

We have proved that any obstruction coming from $\ell_dt_{d}^{b_d}$ has the form given in the statement of the proposition. Moreover, this form makes it impossible for such an obstruction to interfere with the failures for $\lambda_{d-1}$ to be syllable-reduced that come from the inductive hypothesis. This proves the first statement. The statement about $\lambda'$ is an immediate consequence. This completes the proof.
\end{proof}

To clarify between $A_\Gamma$ and $A_\Lambda$, write $|\lambda|_{T^*}$ for the length of a word $\lambda$ in the alphabet $T$, and write $|g|_{A_\Gamma}$ for the word-length of an element $g\in A_\Gamma$. We shall use the term  \emph{$A_\Gamma$-prefix} to clarify that we are taking a prefix of an element $g\in A_\Gamma$. Combining Item~\ref{sh:rep_F_change} with the result of Proposition~\ref{prop:roots_structure} allows us to find a pair of closely related representatives for the elements $F(g)$ and $F(h)$ of $A_\Lambda$. 

\begin{lemma} \label{lem:prefixes}
Let $g\in A_\Gamma$, and let $h\in A_\Gamma$ be an $A_\Gamma$-prefix of $g$. There is a word $\mu$ representing $F(g)$ such that:
\begin{itemize}
\item   the length $|\mu|_{T^*}$ is at most $|F(g)|_{A_\Lambda}+2$;
\item   $\mu$ has an initial subword $\nu$ that minimally represents $F(h)$;
\item   if $|\mu|_{T^*}>|F(g)|_{A_\Lambda}$, then $h$ has a unique final letter, and the first letter of $\mu$ after $\nu$ is the inverse of the last letter of $\nu$;
\end{itemize}
\end{lemma}

\begin{proof}
The idea of the proof is as follows. If we take a lazy representative $\gamma$ of $g$, which exists by Lemma~\ref{lem:lazy_exists}, then Proposition~\ref{prop:roots_structure} gives a good representative $\lambda$ of $F(g)$. We would like to say that $h$ is represented by an initial substring of $\gamma$, which would necessarily be a lazy word, for then we would get an initial substring of $\lambda$ representing $F(h)$. However this need not be the case. In general, a representative of $h$ can be found in an initial substring of $\gamma$, but ``padded out'' by some additional syllables. These excess syllables have to commute with the ones making up $h$ that appear after them, because $h$ is a prefix of $g$. Intuitively, we prove the lemma by shuffling the ``images of the excess syllables'' in $\lambda$ past the ones coming from $h$.

Using Lemma~\ref{lem:lazy_exists}, fix a lazy representative $\gamma=s_{1}^{a_1}\dots s_{d}^{a_d}$ of $g$ and let $F(\gamma)=\lambda=\ell_1t_{1}^{b_1}\dots\ell_dt_{d}^{b_d}$. Recall from Item~\ref{sh:roots} that  $b_{k}$ differs from $a_{k}$ by at most one, for every $k$. The letters making up $\lambda$ come in two types: those that appear in some $\ell_c$, and those that appear in some $t_{c}^{b_c}$. We call these \emph{type-$\L$} and \emph{type-$T$} letters, respectively.

Let $e=\syl h$. As discussed above, whilst $h$ may not be represented by any initial substring of $\gamma$, there is necessarily some $r\ge0$ such that $h$ can be represented by a word obtained from $\gamma_{e+r}$, the first $e+r$ syllables of $\gamma$, by deleting $r$ syllables $s_{x}^{a_x}$ other than the final one (and possibly shrinking the powers of the final two non-deleted syllables). We call such values of $x$ and the corresponding syllables \emph{excess}. The fact that $h$ is an $A_\Gamma$-prefix of $g$ implies that every excess syllable must commute with every non-excess syllable that appears after it in $\gamma_{e+r}$. 
By the discussion in Item~\ref{sh:rep_F_change}, the excess type-$T$ letters $t_{x}$ commute with $\ell_yt_{y}^{b_y}$ for every non-excess $y\in\{x+1,\dots,e+r\}$. Let $\eta$ be the word representing $F(g)$ obtained by shuffling every excess $t_{x}^{b_x}$ to appear after $\ell_yt_{y}^{b_y}$ for every non-excess $y\in\{x+1,\dots,e+r\}$.
It has an initial subword that represents $F(h)$.

Note that $|\eta|_{T^*}=|\lambda|_{T^*}$. If this length is equal to $|F(g)|_{A_\Lambda}$, then we can set $\mu=\eta$ and we are done: the initial subword representing $F(h)$ is necessarily minimal because it is contained in a word with no possible cancellation.

Otherwise, since $\gamma$ is lazy, Proposition~\ref{prop:roots_structure} tells us exactly how $\lambda$ fails to be syllable-reduced. In particular, there is at least one pair $(c,k)$ as in the statement of Proposition~\ref{prop:roots_structure} causing $\lambda$ to fail to be syllable-reduced. If $(c_1,k_1)$ and $(c_2,k_2)$ are two such pairs, then, after relabelling, the properties defining them force $c_1+k_1<c_2$. Thus there can be at most one pair $(c_0,k_0)$ for which $e+r\in\{c_0,\dots,c_0+k_0\}$. 
Note that since $t_{c},\dots,t_{{c+k}}$ all commute with exactly one of the last two letters of $\ell_c$, which form a shortcut pair, two of these $t_{{c+x}}$ can commute only if they are equal, which requires them to be non-consecutive because we started with a syllable-reduced word representing $g$. In particular, if there is such a pair, then $h$ has a unique final letter and $t_{{e+r}}^{\pm1}$ is the unique final type-$T$ letter of $F(h)$.

Let $\lambda'$ be the syllable-reduced word output by Proposition~\ref{prop:roots_structure}, which is obtained from $\lambda$ by shuffling and cancelling type-$\L$ letters coming from pairs $(c,k)$ as in that proposition. 
Let $\lambda''$ be the word obtained from $\lambda$ by doing all of these type-$\L$ shuffles and cancellations except for the one corresponding to the pair $(c_0,k_0)$, if it exists. Note that $|\lambda''|_{T^*}=|\lambda'|_{T^*}$ or $|\lambda'|_{T^*}+2$. According to Proposition~\ref{prop:roots_structure}, the length of $\lambda''$ is either $|F(g)|_{A_\Lambda}$ or $|F(g)|_{A_\Lambda}+2$.

In the construction of the word $\eta$ from $\lambda$ given above, we shuffled excess type-$T$ letters to appear after some non-excess $\ell_yt_{y}^{b_y}$. Since $\lambda''$ was obtained from $\lambda$ by deleting some type-$\L$ letters, we can define $\eta''$ to be the word obtained from $\lambda''$ by shuffling type-$T$ letters in $\lambda''$ in exactly the same way. We have $|\eta''|_{T^*}=|\lambda''|_{T^*}$, and $\eta''$ represents $F(g)$.

Let the final syllable of $h$ be $s_{{e+r}}^{a'_{e+r}}$, which is a prefix of $s_{{e+r}}^{a_{e+r}}$ with $a'_{e+r}\ne0$. Let $\delta\in\{-1,0,1\}$ be such that $b_{e+r}=a_{e+r}+\delta$. Set $b'_{e+r}=a'_{e+r}+\delta$. 

If $|\eta''|_{T^*}=|F(g)|_{A_\Lambda}$ then set $\mu=\eta''$. In this case, we define $\nu$ to be the initial subword of $\mu$ that ends with $t_{{e+r}}^{b'_{e+r}}$.

Otherwise, the way that $\eta''$ fails to have minimal word-length is that the label $\ell_{c_0}$ ends with letters $\tau_1^\epsilon\tau_2^\epsilon$, where $\epsilon\in\{\pm1\}$, and $\ell_{c_0+k_0+1}$ begins with the letters $\tau_1^{-\epsilon}\tau_2^{-\epsilon}$, and there is extra information given by Proposition~\ref{prop:roots_structure} (in particular, $\tau_1$ commutes with the intervening letters $t_{c_0},\dots,t_{{c_0+k_0}}$ but $\tau_2$ does not). 

This time, we first define $\nu$ from $\eta''$ and then use it to define $\mu$. Starting from $\eta''$, shuffle the letter $\tau_1^\epsilon$ coming from $\ell_{c_0}$ to appear after every (necessarily non-excess) $t_{y}^{b_y}$ between it and $t_{{e+r}}^{b_{e+r}}$. Then shuffle $\tau_1^\epsilon$ one more time, to appear after $t_{{e+r}}^{b'_{e+r}}$. Let $\nu$ be the initial subword of the word obtained in this way that ends with this $\tau_1^\eps$. 

We perform one more shuffle to produce $\mu$. Note that every excess $t_{x}$ commutes with $\tau_1$. Indeed, if $x\in\{c_0,\dots,c_0+k_0\}$, then this is directly given by Proposition~\ref{prop:roots_structure}, whereas if $x<c$ then it is because $t_{x}$ commutes with $\ell_c$, as given by Item~\ref{sh:rep_F_change}. To complete the definition of $\mu$, shuffle the letter $\tau_1^{-\epsilon}$ coming from $\ell_{c_0+k_0+1}$ to the left until it appears immediately after the last letter of $\nu$. This only involves shuffling it past syllables that are either from the set $\{t_{{c_0}},\dots,t_{{c_0+k_0}}\}$ or are excess, and we just showed that $\tau_1$ commutes with all such syllables. 

The word $\mu$ has a single failure to be reduced as a word: the final letter of its initial subword $\nu$ is $\tau_1^\epsilon$, which cancels with the very next letter of $\mu$. 
By the commutation properties of $\tau_1$ established above, the word $\mu$ was obtained from $\eta''$ by shuffling commuting elements and not cancelling anything, so $\mu$ represents $F(g)$ and has the same length as $\eta''$. 

It remains to show that $\nu$ represents $F(h)$, and does so minimally. As noted earlier, we know that $\eta$ has an initial subword representing $F(h)$. We can think of $\eta''$ as being obtained from $\eta$ by making some cancellations of pairs of type-$\L$ letters. Since we only shuffled some excess type-$T$ letters to go from $\lambda$ to $\eta$, the construction of $\eta''$ ensures that for each type-$\L$ pair that was cancelled, either both letters appear before $t_{{e+r}}$ inside $\eta$, or both appear after $t_{{e+r}}$ inside $\eta$. It follows that $\eta''$ also has an initial subword $\nu''$ that represents $F(h)$, namely the initial subword that ends with $t_{{e+r}}^{b'_{e+r}}$.

If $\mu=\eta''$, then $\nu''$ is equal to $\nu$, and $\nu$ is minimal because it is a subword of the word $\mu$ which has no possible cancellations. Otherwise, $\nu$ is obtained from $\nu''$ by shuffling $\tau_1^\epsilon$ to the end, and $\tau_1^\epsilon$ commutes with all letters it is shuffled past. Thus $\nu$ represents $F(h)$. By construction, it ends with only one of the only two letters of $\mu$ that are cancellable, so we have $|F(h)|_{A_\Lambda}=|\nu|_{T^*}$ as desired.
\end{proof}

We are now in a position to prove Proposition~\ref{prop:building_cycles_qie}, the main result of this section.

\begin{proposition} \label{prop:building_cycles_qie}
Let $n>6$, and let $m=n+p(n-4)+q(n-2)$ for some $p\ge1$, $q\ge0$. The map $F\colon A_{C_m}\to A_{C_n}$ constructed above is a quasiisometric embedding.
\end{proposition}

\begin{proof}

First we show that $F$ is Lipschitz. By the triangle inequality, it suffices to upper-bound $\dist(F(g),F(gs))$, where $g\in A_\Gamma$ and $s$ is a generator or an inverse of a generator such that $|gs|>|g|$. By the definition of $F$, if $s$ is equal to a last letter of $g$, then $\dist(F(g),F(gs))=1$. Otherwise we can apply Proposition~\ref{prop:roots_structure} to see that $\dist(F(g),F(gs))\le\max\{|\ell|\,:\,\ell\in\L\}+2$. Since all elements of $\L$ have length at most $p+2q$, this proves that $F$ is $(p+2q+2)$--Lipschitz.

It remains to show that $F$ is coarsely colipschitz. Let $g_1,g_2\in A_\Gamma$, and let $g_0\in A_\Gamma$ be their maximal common $A_\Gamma$-prefix. According to Lemma~\ref{lem:prefixes}, the elements $F(g_1)$ and $F(g_2)$ have a common $A_\Lambda$-prefix $P_0$ that is either equal to $F(g_0)$ or can be obtained from $F(g_0)$ by deleting a final letter. Let $P$ be the maximal common $A_\Lambda$-prefix of $F(g_1)$ and $F(g_2)$, which may be longer than $P_0$. We aim to show that $P$ is not much longer than $P_0$. 

For $k\in\{1,2\}$, let $\mu_k$ and $\nu_k$ be the representatives of $F(g_k)$ and $F(g_0)$ given by applying Lemma~\ref{lem:prefixes} to $g_k$ and its prefix $g_0$. Let $\bar\mu_k$ be the minimal representative for $F(g_k)$ that is obtained from $\mu_k$ by doing at most one cancellation, and let $\bar\nu_k$ be obtained from $\nu_k$ by deleting its last letter if $\bar\mu_k\ne\mu_k$. In Item~\ref{sh:rep_F_change} we saw how to compare the terms coming from $\L$ in representatives of $F(g_k)$ obtained by applying the definition of $F$ to different representatives of $g_k$. That discussion in particular shows that there cannot be any syllable $t_j^b$ such that $P_0t_j^b$ is a common $A_\Lambda$-prefix of $F(g_1)$ and $F(g_2)$. Indeed, we would be able to shuffle such a syllable to appear immediately after the syllable of $\bar\mu_k$ with which $\bar\nu$ ends, and then the discussion of Item~\ref{sh:rep_F_change} would tell us that $g_1$ and $g_2$ have a common $A_\Gamma$-prefix consisting of $g_0$ followed by another syllable, contradicting maximality of $g_0$. It follows that $P$ either has the form $P_0\ell$ for some $\ell\in\L$, or can be obtained from such an expression by deleting a final letter. In particular, it is at most $p+2q$ letters longer than $P_0$.

Now we use Lemma~\ref{lem:prefixes} to lower-bound $\dist(F(g_k),F(g_0))$. Let $\mu'_k$ be the terminal substring obtained from $\mu_k$ by deleting $\nu_k$. Since $\nu_k$ is a representative of $F(g_0)$ and $\mu_k$ is a representative of $F(g_k)$, it follows that $\mu'_k$ is a representative of $F(g_0)^{-1}F(g_k)$. By construction, at most one letter that is cancellable in $\mu_k$ appears in $\mu'_k$, and it follows that $\mu'_k$ is a minimal representative of $F(g_0)^{-1}F(g_k)$. In other words, $\dist(F(g_0),F(g_k))=|\mu'_k|_{T^*}$. We therefore wish to lower-bound the length of the word $\mu'_k$. To do so, we examine the construction of $\mu_k$.

To produce $\mu_k$, we let $\gamma_k$ be a lazy representative of $g_k$, which gives rise to a representative $\lambda_k$ of $F(g_k)$ as in Item~\ref{sh:roots}. The construction of $\mu_k$ in the proof of Lemma~\ref{lem:prefixes} starts with $\lambda_k$, does all or all-but-one of the cancellations given by Proposition~\ref{prop:roots_structure}, and then shuffles all \emph{excess} syllables of \emph{type-$T$} past the last letter of $\nu_k$ inside $\lambda_k$. In particular, for each syllable $s^a$ that appears in $g_k$ but does not form part of $g_0$, there is a corresponding syllable $t^b$ that appears in $\mu_k$ after the end of $\nu_k$, where $b\in\{a,a+1\}$ if $a>0$ and $b\in\{a-1,a\}$ if $a<0$. This shows that $\mu'_k$ has length at least $|g_0^{-1}g_k|$.

Combining these estimates, we compute
\begin{align*}
\dist(F(g_1),F(g_2)) \,&=\, \dist(F(g_1),P) + \dist(P,F(g_2)) \\
    &\ge\,  \dist(F(g_1),F(g_0)) + \dist(F(g_0),F(g_2)) -2(p+2q) \\
    &\ge\,  \dist(g_1,g_0) + \dist(g_0,g_2) -2p-4q \\
    &=\, \dist(g_1,g_2)-2p-4q.
\end{align*}
This completes the proof.
\end{proof}

\begin{remark}
The proof of Proposition~\ref{prop:building_cycles_qie} showed the slightly stronger fact that $F$ is coarsely colipschitz with multiplicative constant 1.
\end{remark}

As noted previously, the map $F\colon A_\Gamma\to A_\Lambda$ induces the map $f\colon\Gamma^\ext\to\Lambda^\ext$, so one consequence of Proposition~\ref{prop:building_cycles_qie} is that $f$ is injective.

In the case where $A_{C_m}$ is not a subgroup of $A_{C_n}$, it is clear that $F$ cannot be at finite distance from any homomorphism. However, even when it does appear as a subgroup, the above construction can yield a quasiisometric embedding that is not close to a homomorphism.

\begin{lemma}\label{lem:no_homo}
If $q>0$, then the quasiisometric embedding $F$ is not at finite distance from a homomorphism.
\end{lemma}

\begin{proof}
If $F$ were at finite distance from a homomorphism, then by \cite[Lem.~7.3]{baderbensaidpetyt:quasiisometric:rigidity}, for each $v\in\Gamma$ there would have to be some $w\in\Lambda$ such that $f(g\cdot v)\in A_\Lambda\cdot w$ for all $g\in A_\Gamma$. The fact that this fails essentially follows from the definition of glides: see Figure~\ref{fig:early_glide}.

To be concrete, using the cyclic ordering on the vertices of $\Gamma=C_m$ and $\Lambda=C_n$, we have $f(v_{2p-2})=\ell\cdot w_{0}$ and $f(v_{2p+1})=\ell t_{\pm 2}\cdot w_{\mp1}$, which is a shortcut of the block $f(\Gamma)$. By definition of the glide along $f(v_{2p+1})$ we have $f(s_{2p+1}\cdot v_{2p-2})=\ell t_{\pm 2}t_{\mp1}^2\cdot w_{\pm2}$, which is not in $A_\Lambda\cdot w_0$.
\end{proof}




On the other hand, when $q=0$ the map $F$ coarsely recovers \cite[Thm~1.12]{kimkoberda:embedability} for $m,n>6$; \emph{cf.}\ \cref{sec:induced_implies_subgroup}.

\begin{lemma} \label{lem:homo}
If $q=0$, then the map $F$ defined above is at finite distance from a homomorphism $A_{C_m}\to A_{C_n}$.
\end{lemma}

\begin{proof}
If $q=0$, then from Item~\ref{sh:doubling} we see that every block appearing as $f(g\cdot \Gamma)$ for some $g\in A_\Gamma$ has the form $B(0,+,0;h)$: there is no shifting or rotation. In particular, if we fix $i$ then $f(g\cdot v_i)$ has the same type for every $g\in A_\Gamma$. Consider the map $\phi$ defined by setting $\phi(g)=h$. That is, $f(g\cdot \Gamma)=B(0,+,0;\phi(g))$. We have $\dist(F,\phi)\le\max\{|\ell|\,:\,\ell\in\L\}=p$. We show that $\phi$ is a homomorphism.

Let $gs_{1}^{a_1}\in A_\Gamma$ be a syllable-reduced expression with $a_1\ne0$, and let $h=\phi(g)$. Let $s_{2}$ be a generator of $A_\Gamma$. The fact that there is no shifting or rotation implies that if $\ell_k\in\L$ is minimal such that $f(v_{k})=\ell_k\cdot w_{k}$ for $k\in\{1,2\}$, then $\ell_1$ is also the minimal element of $\L$ such that $f(g\cdot v_{1})=\phi(g)\ell_1\cdot w_{1}$, and $\ell_2$ is the minimal element of $\L$ such that $f(gs_{1}^{a_1})\cdot v_{2}=\phi(gs_{1}^{a_1})\ell_2\cdot w_{2}$. By definition, we therefore have 
\begin{align*}
\phi(gs_{1}^{a_1}s_{2}^{a_2}) \,=\, \phi(g)\ell_1t_{1}^{a_1}\ell_1^{-1}\ell_2t_{2}^{a_2}\ell_2^{-1}
    \,=\, \phi(gs_{1}^{a_1})\phi(s_{2}^{a_2}),
\end{align*}
and it follows that $\phi$ is a homomorphism.
\end{proof}

\section{Cyclic subgraphs of extension graphs} \label{sec:cyclic_subgraphs_of_ext_graphs}

In the previous section we constructed a quasiisometric embedding $A_\Gamma\to A_\Lambda$ by first constructing a combinatorial embedding of their extension graphs. By \cref{thm:qie_induces}, the existence of a combinatorial embedding $\Gamma^\ext\to\Lambda^\ext$ is a necessary condition for the existence of such a quasiisometric embedding, since $\Gamma$ and $\Lambda$ are cycles of length greater than four. 

In this section we show that when $n>4$, if $m>n$ cannot be written as $n+p(n-4)+q(n-2)$ for $p\geq 1$, then $C_m^\ext$ does not combinatorially embed in $C_n^\ext$, and thus by \cref{thm:qie_induces} there does not exist a quasiisometric embedding $A_{C_m}\to A_{C_n}$.

We will need the following tool:

\begin{lemma} \label{lem:gen_hom_in_extgraph}
    Let $\Gamma$ be a triangle-free graph. Identify $\Gamma$ with $1\cdot \Gamma\subseteq \Gamma^\ext$. If $\{\sigma_i\, :\, i\in I\}$ is a set of $1$--cycles that are representatives of a generating set of $H_1(\Gamma,\sfrac\Z{2\Z})$, then the homology classes of $\{g\cdot\sigma_i\, :\, i\in I,\, g\in A_{\Gamma}\}$ generate $H_1(\Gamma^{\ext},\sfrac\Z{2\Z})$. In particular, if $\Gamma=C_n$ then $H_1(\Gamma,\sfrac\Z{2\Z})$ is generated by the homology classes of $\{g\cdot C_n\, :\, g\in A(C_n)\}$.
\end{lemma}

\begin{proof}
    By \cite[Lemma 3.1]{kimkoberda:embedability} (see also \cite{kimkoberdalee:finite}), given a graph $\Gamma$ and a finite subgraph of $\Gamma^{\ext}$ denoted $\Lambda$, there is a sequence $1\cdot \Gamma=\Gamma_0\leq \Gamma_1\leq \dots \leq \Gamma_l$ of induced subgraphs of ${\Gamma^{\ext}}$ and vertices $g_i\cdot v_{s_i}\in \Gamma_{i-1}$ such that:
    \begin{itemize}
        \item $\Gamma_i=\Gamma_{i-1}\cup g_is_ig_i^{-1}\cdot \Gamma_{i-1}$ as a subgraph of $\Gamma^{\ext}$ and is isomorphic to the abstract double $\Gamma_{i-1}\cup_{\str_{\Gamma_{i-1}}(v_i)} \Gamma_{i-1}$, and
        \item $\Lambda$ is a subgraph of $\Gamma_l$.
    \end{itemize}
    We will call such a sequence a \emph{doubling sequence of length $l$}.

Let $\{\sigma_i\, :\, i\in I\}$ generate $H_1(\Gamma,\sfrac\Z{2\Z})$. We will prove by induction on $l$ that if $\sigma\in H_1(\Gamma^\ext,\sfrac\Z{2\Z})$ and there exists a doubling sequence of length $l$ containing the support of $\sigma$, then there exists $w_1,...,w_k$ such that $\sigma=\sum_{j=1}^k w_j\cdot \sigma_{i_j}$.
As every element of $H_1(C_n^{\ext},\sfrac\Z{2\Z})$ has finite support, this will establish the result.

If $l=1$ there is nothing to prove. For $l>1$, let $\sigma$ be such that there exists doubling sequence of length $l$ such that $\Gamma_l$ contains the support of $\sigma$. Let $g_l\cdot v_{s_l}\in \Gamma_{l-1}$ be the vertex along which we doubled to obtain $\Gamma_l$ from $\Gamma_{l-1}$. Since $\Gamma$ is triangle-free, $\str(v_l)$ is contractible. By \cite[Lem. 3.9(i)]{kimkoberda:embedability}, so is $\Gamma^\ext$. Hence $H_1(\Gamma_l,\sfrac\Z{2\Z})$ is isomorphic to $H_1(\Gamma_{l-1}\vee \Gamma_{l-1},\sfrac\Z{2\Z})$. The image of $\sigma$ via this isomorphism is a sum of two cycles, each supported on a copy of $\Gamma_{l-1}$, hence $\sigma=\sigma^1+\sigma^2$ where $\sigma^1$ is a cycle supported on $\Gamma_{l-1}$ and $\sigma^2$ is a cycle supported on $g_ls_lg_l^{-1}\cdot \Gamma_{l-1}$. By the inductive hypothesis applied to $\sigma^1$ and $(g_ls_lg_l^{-1})^{-1}\cdot \sigma^2$, there exist $w_1,...,w_k,h_1,\dots, h_r\in A_{\Gamma}$ such that $\sigma^1=\sum_{j=1}^kw_j\cdot \sigma_{i_j}$ and $(g_ls_lg_l^{-1})^{-1}\cdot \sigma^2=\sum_{j=1}^mh_j\cdot \sigma_{i'_j}$. This shows $\sigma=\sum_{j=1}^kw_j\cdot \sigma_{i_j}+\sum_{j=1}^m(g_ls_lg_l^{-1})h_j\cdot \sigma_{i'_j}$, which finishes the proof. 
\end{proof}

The following is the main result of this section.

\begin{proposition} \label{prop:cycles_equation}
Let $m,n>4$. 
If $C_m^\ext$ combinatorially embeds in $C_n^\ext$, then either $m=n$ or there exist $p\geq 1,q\ge0$ such that $m=n+p(n-4)+q(n-2)$. 
\end{proposition} 

\cref{prop:cycles_equation} follows immediately from the next two lemmas.

\begin{lemma}
\label{lem:embedding_of_cycle_in_ext}
Let $m,n>4$. 
If $C_m$ combinatorially embeds in $C_n^\ext$, then there exist $p,q\ge0$ such that $m=n+p(n-4)+q(n-2)$. 
\end{lemma} 

\begin{proof}
We view cyclic subgraphs of $C_n^\ext$ as elements of $H_1(C_n^\ext,\sfrac\Z{2\Z})$, which is generated by $\{g\cdot C_n\,:\,g\in A_{C_n}\}$, by \cref{lem:gen_hom_in_extgraph}.

For an element $\sigma\in H_1(C_n^\ext,\sfrac\Z{2\Z})$, we write $M(\sigma)$ for the number of edges in its support. If $\sigma=\sum_{i=1}^kg_i\cdot C_n$, we can consider $\hat\sigma$ to be the element of $H_1(C_n^\ext,\Z)$ represented by $\sum_{i=1}^kg_i\cdot C_n$ where we choose an orientation of $C_n$ and that fixes an orientation on $g\cdot C_n$. Observe that the sum of the coefficients of edges in $\hat\sigma$ is $kn$, hence $M(\sigma)$ has the same parity as $kn$.

Let us show by induction that if $\sigma\in H_1(C_n^\ext,\sfrac\Z{2\Z})$ is a sum of $k+1$ copies of $C_n$, then $M(\sigma)\ge n+k(n-4)$. For $k=0$ this is trivial, so suppose $k>0$. Let $\Sigma\subset A_{C_n}$ be such that $\sigma=\sum_{g\in\Sigma}g\cdot C_n$. Since $\Sigma$ is finite, there is an element $g\in \Sigma$ with maximal syllable-length. Let $\sigma'=\sigma-g\cdot C_n$. By the inductive hypothesis, $M(\sigma')\ge n+(k-1)(n-4)$. We split into two cases, according to whether or not $g$ has a unique final syllable.

If $g$ has a unique final syllable, then let $g=g't^a$ be a syllable-reduced expression for $g$ where $t$ is a standard generator of $A_{C_n}$. In particular, if $w$ is the vertex of $1\cdot C_n$ of type $t$, then $g'\cdot w=g\cdot w$. 
Suppose $h\in \Sigma$ is such that $h\cdot C_n$ shares at least one edge with $g\cdot C_n$. If there were such an edge that did not contain $w$, then $h$ would have greater syllable-length than $g$, which is impossible. This shows that at least $n-2$ edges of $g\cdot C_n$ lie outside the support of $\sigma'$. Hence $M(\sigma)\ge M(\sigma')+n-4$, where the ``$-4$'' arises because the other two edges may ``cancel out'' two of the nontrivial coefficients of $\sigma'\in H_1(C_n^\ext,\sfrac\Z{2\Z})$.

If $g$ does not have a unique final syllable, then since $C_n$ is triangle-free, there is a syllable-reduced expression $g=g't_1^at_2^b$ where $t_1$ and $t_2$ are commuting standard generators of $A_{C_n}$, neither commutes with any final letter of $g'$, and $a$ and $b$ are nonzero. Let $w_1$ and $w_2$ be the vertices of $1\cdot C_n$ corresponding to $t_1$ and $t_2$, respectively. If $h\in \Sigma$ is such that $h\cdot C_n$ shares at least one edge with $g\cdot C_n$, then the maximality of $g$ implies that, up to swapping $t_1$ and $t_2$, there are only three possibilities for $h$: either $h=g't_1^c$; or $h=g't_1^ct_2^d$; or $h=g't_1^at_2^d$ - in these expressions, $c\not\in\{0,a\}$, $d\not\in\{0,b\}$. See Figure~\ref{fig:structure_of_C_m}. In particular, this shows that at most three edges of $g\cdot C_n$ can lie in the support of $\sigma'$, namely the edge $g\cdot w_1w_2$ and its two neighbours $g\cdot w_0w_1$ and $g\cdot w_2w_3$. If at most two edges of $g\cdot C_n$ lie in the support of $\sigma'$, then as above we deduce that $M(\sigma)\ge M(\sigma')+n-4$. 

\begin{figure}[ht]
\includegraphics{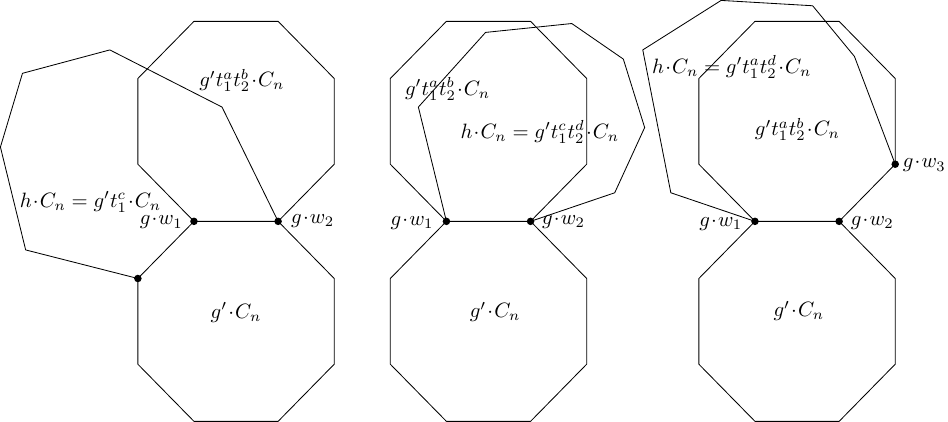}\centering 
\caption{The three possible ways for $h\cdot C_n$ to share an edge with $g\cdot C_n$ when $g$ does not have a unique final syllable.} \label{fig:structure_of_C_m}
\end{figure}

Suppose that all three edges $g\cdot w_1w_2$, $g\cdot w_0w_1$, and $g\cdot w_2w_3$ lie in the support of $\sigma'$, and in particular are not in $\sigma$. The description of the possible copies of $C_n$ attached to $g\cdot C_n$ tells us that there is some $h\in \Sigma\ssm\{g\}$ of the form $h=g't_1^at_2^d$. Let $\sigma''=\sigma'-h\cdot C_n$. We note that $h$ also has maximal syllable length in $\Sigma$, and so by the same analysis as in the previous paragraph, the only edges of $h\cdot C_n$ that can lie in the support of $\sigma''$ are $h\cdot w_0w_1$, $h\cdot w_1w_2$, and $h\cdot w_2w_3$. But the assumption that $g\cdot w_1w_2=h\cdot w_1w_2$ and $g\cdot w_2w_3=h\cdot w_2w_3$ both lie in the support of $\sigma'$ implies that neither lies in the support of $\sigma''$, as we have $\sfrac\Z{2\Z}$ coefficients. Hence 
\[
M(\sigma) \,=\, M(\sigma')+n-6 \,=\, M(\sigma'')+(n-6)+(n-2)
\]
and so satisfies the desired inequality, by induction.

Now assume that $C_m$ combinatorially embeds in $C_n^{\ext}$ with image $C$. Let $\sigma\in H_1(C_n^\ext,\sfrac\Z{2\Z})$ correspond to $C$. The nonzero class $\sigma$ is a sum of $k+1$ copies of $C_n$, for some $k\ge0$. We claim that $M(\sigma)\le n+k(n-2)$. This is true if $k=0$. For $k>0$, if this were not the case, then there would be some $g\cdot C_n$ making up $\sigma$ such that $M(\sigma-g\cdot C_n)=M(\sigma)-n$, by induction. But for this to occur, $g\cdot C_n$ would have to share no edges with the support of $\sigma-g\cdot C_n$, which is impossible, since $C$ is a cyclic subgraph.
We have shown that $M(\sigma)$ lies between $n+k(n-4)$ and $n+k(n-2)$, and has the same parity as $nk$, so the first statement follows.
\end{proof}

\begin{lemma}
\label{lem:embedding_of_extcycle_in_ext}
Let $m,n>4$. 
If $m>n$ is such that in every expression $m=n+p(n-4)+q(n-2)$ with $p,q\ge0$ we must have $p=0$, then $C_m^{\ext}$ does not combinatorially embed in $C_n^{\ext}$. 
\end{lemma}

\begin{proof}
Assume towards a contradiction that we have an embedding of $C_m^{\ext}$ in $C_n^{\ext}$ and that if $m=n+p(n-4)+q(n-2)$ then $p=0$ and $q>0$. By considering the number $(n-2)(n-4)$, we note that we must have $q<n-4$.

Let $C$ be the image of $1\cdot C_m$ and let $\sigma=\sum_{g\in \Sigma} g\cdot C_n\in H_1(C_n^\ext,\sfrac\Z{2\Z})$ correspond to $C$. Denote by $\Sigma$ the support of $\sigma$. We have $m=M(\sigma)=n+q(n-2)$.
As we have seen, $n+(|\Sigma|-1)(n-4)\leq M(\sigma)\leq n+(|\Sigma|-1)(n-2)$, so we must have $|\Sigma|=q+1$ because $q<n-4$.

A counting argument gives us that every vertex in $\bigcup_{g\in \Sigma}g\cdot C_n$ is in $C$. 
Indeed, $C$ contains $m=n+q(n-2)$ vertices, while each $g\cdot C_n$ contains $n$ vertices, but for every copy of $C_n$ at least two of the vertices are counted twice, as $g\cdot C_n$ must intersect $\bigcup_{g'\in \Sigma-\{g\}}g'\cdot C_n$ on at least an edge (otherwise $C$ contains a cut-vertex or is disconnected), so the number of vertices in $\bigcup_{g\in \Sigma}g\cdot C_n$ is at most $n(q+1)-2q=n+q(n-2)$. Thus, there is equality.

Consider all $q$ edges in $\bigcup_{g\in \Sigma} g\cdot C_n$ that are not in $C$. We must have a vertex $w\in C$ that is the endpoint of exactly one such edge $e$: if not then we get a cycle of such edges, contradicting the fact that $q<n-4$.

We see that $w$ lies in exactly two elements of $\{g\cdot C_n\,:\,g\in \Sigma\}$. 
We denote by $w_0,\dots, w_{n-1}$ the vertices of $1\cdot C_n$ in cyclic order and we denote $t_i=\type(w_i)$.
Up to the left action of $A_{C_n}$ on $C_n^{\ext}$ and reversing the ordering of $w_0,\dots,w_{n-1}$, we may assume $w=w_1$ and the other endpoint of $e$ is $w_0$.
The neighbours of $w$ in $C$ are $w_2$ and $g\cdot w_2=t_0^b\cdot w_2$. See Figure~\ref{fig:doubling_C_m}.

\begin{figure}[ht]
\includegraphics{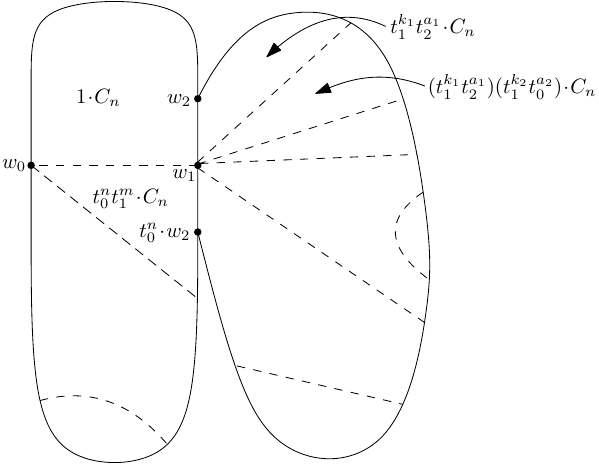}\centering 
\caption{The subgraphs $C$ on the left and $C'$ on the right, intersecting in $\{w_2,w_1,t_0^n\cdot w_2\}$.} \label{fig:doubling_C_m}
\end{figure}

As $C$ is part of the image of a combinatorial embedding of $C_m^{\ext}$ in $C_n^\ext$, there is an image $C'$ of another copy of $C_m$ that intersects $C$ exactly on the path $(w_2,w_1,t_0^n\cdot w_2)$. See Figure~\ref{fig:doubling_C_m}.
As explained for $C$, the cycle $C'$ corresponds to some $\sigma'=\sum_{i=1}^{q+1} h_i\cdot C_n\in H_1(C_n^{\ext},\sfrac{\Z}{2\Z})$. Up to renumbering the $h_i$, we may assume that $(w_1,w_2)\in h_1\cdot C_n$, that $(w_1,t_0^n\cdot w_2)\in h_k\cdot C_n$, and that $h_{i+1}\cdot C_n$ intersects $h_{i}\cdot C_n$ on an edge containing $w_1$ for each $i<k$. This can be seen by considering all edges in $\bigcup_{i=1}^{q+1} h_i\cdot C_n$ containing $w_1$ that are not in $C'$. Note that we must have $k>1$, because no translate of $C_n$ inside $C_n^\ext$ has two vertices of type $t_2$.

As $C$ contains all vertices of $\bigcup g_i\cdot C_n$, we have $(h_1\cdot C_n)\cap(1\cdot C_n)=(w_1,w_2)$, hence $h_1=t_1^{k_1}t_2^
{a_1}$ for $k_1,a_1\neq 0$. 
We then have by induction that $h_\ell = (t_1^{k_1}t_2^{a_1})(t_1^{k_2}t_0^{a_2})(t_1^{k_3}t_2^{a_3})\cdots(t_1^{k_\ell}t_\epsilon^{a_\ell})$ for $a_i,k_i\neq 0$, where $\epsilon\in \{0,2\}$ depends on the parity of $\ell$.

By definition, $t_0^n\cdot w_2\in h_k\cdot C_n$ and so $h_k\cdot w_2=t_0^n\cdot w_2$. In other words, $h_kt_2h_k^{-1}=t_0^nt_2t_0^{-n}$. From the above description of $h_\ell$, we have $h_kt_2h_k^{-1}=(t_2^{a_1}t_0^{a_2}\cdots t_\epsilon^{a_k})t_2(t_2^{a_1}t_0^{a_2}\cdots t_\epsilon^{a_k})^{-1}$.
If $\epsilon=0$ then this expression is reduced and cannot equal $t_0^nt_2t_0^{-n}$. If $\epsilon=2$ then we have $h_kt_2{h_k^{-1}}=(t_2^{a_1}t_0^{a_2}\cdots t_0^{a_{k-1}})t_2(t_2^{a_1}t_0^{a_2}\cdots t_0^{a_{k-1}})^{-1}$ and this expression is reduced and cannot equal $t_0^nt_2t_0^{-n}$.
\end{proof}

We get as a corollary the ``only if'' direction of \cref{mthm:cycles_qie}.

\begin{corollary} \label{cor:necessary}
Let $m,n\geq 4$. If there exists a quasiisometric embedding $A_{C_m}\to A_{C_n}$, then either $m=n$ or there exists $p\geq 1$, $q\geq 0$ such that $m=n+p(n-4)+q(n-2)$.
\end{corollary}

\begin{proof}
For $n=4$, \cite[Thm~5.8]{baderbensaidpetyt:quasiisometric:rigidity} shows that $m$ cannot be odd and all even $m\geq 4$ numbers are of the form $n$ or $n+(n-4)+q(n-2)$. 

If $n>4$, then $A_{C_4}$ cannot quasiisometrically embed in $A_{C_n}$, by \cite[Cor.~5.9]{baderbensaidpetyt:quasiisometric:rigidity}.
If $m,n>4$, then $C_n$ and $C_m$ are triangle-free and square-free, so \cref{thm:qie_induces} shows that any quasiisometric embedding $A_{C_m}\to A_{C_n}$ induces a combinatorial embedding $C_m^\ext\to C_n^\ext$. By \cref{prop:cycles_equation}, this implies that $m=n$ or $m=n+p(n-4)+q(n-2)$ for some $p\ge1$, $q\ge0$.
\end{proof}

\section{Proof of Theorem~\ref{mthm:cycles_qie}} \label{sec:qie_between_cycles}

In this short section, we combine the results of Sections~\ref{sec:building_qie} and~\ref{sec:cyclic_subgraphs_of_ext_graphs} to prove \cref{mthm:cycles_qie}, which we restate for convenience.

\numberedtheorem{Theorem~\ref{mthm:cycles_qie}}{}{
Let $m,n\ge 4$. There is a quasiisometric embedding $A_{C_m}\to A_{C_n}$ if and only if either $m=n$ or there exist $p\ge1$, $q\ge0$ such that $m=n+p(n-4)+q(n-2)$.
}

\begin{proof}
Corollary~\ref{cor:necessary} shows that it is necessary for $m$ and $n$ to satisfy the stated relation in order for a quasiisometric embedding $A_{C_m}\to A_{C_n}$ to exist. We show that one does exist in all cases where the relation is satisfied.

First consider the case where $n=4$. In this case, $m$ must be even.
If $m$ is even, then $C_m$ is 2--colourable, so the main result of \cite{rull:embedding} states that $A_{C_m}$ can be quasiisometrically embedded in a product of two finite-valence trees, which is quasiisometric to $F_2\times F_2=A_{C_4}$.

If $n=5$, then it was shown in \cite[Thm~1.12]{kimkoberda:embedability} that $A_{C_m}<A_{C_5}$ for all $m\ge5$, and moreover the proof therein, which uses \cite{clayleiningermangahas:geometry}, shows that $A_{C_m}$ appears as an undistorted subgroup of $A_{C_5}$. Similarly, if $n=6$, then \cite[Thm~1.12]{kimkoberda:embedability} shows that $A_{C_{2k}}$ is an undistorted subgroup of $A_{C_6}$ for all $k>2$. Alternatively, the cases $n\in\{5,6\}$ are both straightforward consequences of \cref{prop:induced_to_homo}.

For $n\ge7$,
Proposition~\ref{prop:building_cycles_qie} shows the existence of a quasiisometric embedding $A_{C_m}\to A_{C_n}$ for any $m$ as in the statement.
\end{proof}





We note that for different values of $p$ and $q$ we obtain quasiisometric embeddings that are not at finite distance from one another, because they correspond to different combinatorial embeddings of extension graphs. 



\section{Exotic quasiisometric embeddings}
\label{sec:exotic_when_q=0}

Let $m,n>4$ and suppose that $m=n+p(n-4)$ for some $p\ge1$. According to \cref{lem:homo}, the quasiisometric embedding $F\colon A_{C_m}\to A_{C_n}$ constructed in \cref{sec:building_qie} (for this decomposition of $m$, with $q=0$ in the notation of Section~\ref{sec:building_qie}) is at finite distance from a homomorphism. In this section we show that, even in this case, there are different quasiisometric embeddings that are not close to homomorphisms. 

It was observed in \cite[\S7]{baderbensaidpetyt:quasiisometric:rigidity} that every nonabelian right-angled Artin group has self-quasiisometries that are not close to homomorphisms (the commensurator has infinite index in the quasiisometry group; \cite[Thm~7.29]{baderbensaidpetyt:quasiisometric:rigidity}), which deals with the case $m=n$. Moreover, \cite[Thm~7.23]{baderbensaidpetyt:quasiisometric:rigidity} gives a complete description of the quasiisometric embeddings $A_{C_n}\to A_{C_n}$, and many of these are not close to homomorphisms. They are nevertheless very rigid. Indeed, every such quasiisometric embedding induces a combinatorial embedding  $C_n^\ext\to C_n^\ext$ with the property that if two vertices have the same type, then their images also have the same type. For $p>0$, we shall construct quasiisometric embeddings $A_{C_m}\to A_{C_n}$ that are neither close to homomorphisms nor have this latter property.

\medskip

As shown in Lemma~\ref{lem:homo}, the map constructed in \cref{sec:building_qie} is close to the homomorphism taking the standard generator $s_i$ to $\type \iota_{+,1}(v_i)$ conjugated by the minimal label of the copy of $\Lambda$ making up the block $B^+(1)$ that it lies in, namely the word $t_2t_{-2}\dots$ of length $\dep (\iota_{+,1}(v_i))$. Denote this homomorphism by $\phi_1$.

A different homomorphism could be obtained by considering $\hat\tau\circ \phi_1\circ \hat\sigma$ where $\hat\sigma:A_{C_m}\to A_{C_m}$ is the automorphism that flips the cyclic order around 0, that is $\hat\sigma(s_i)=s_{-i}$, and similarly $\hat\tau:A_{C_n}\to A_{C_n}$ is defined by $\hat{\tau}(t_j)=t_{-j}$. Denote this homomorphism by $\phi_2$. For instance, we have 
\[
\phi_2(s_0)\,=\, t_0 \,=\, \phi_1(s_0), \quad \phi_2(s_{\pm1}) \,=\, t_{\pm1} \,=\, \phi_1(s_{\pm1}), 
\quad \phi_2(s_2) \,=\, t_2 \,\ne\, \phi_1(s_2).
\]
In fact, we have $\phi_2(s_i)=\phi_1(s_i)$ if and only if $i\in\{-1,0,1\}$. Moreover, $\phi_j(s_i)$ commutes with $t_0$ if and only if $i\in\{-1,0,1\}$.

Define a map $\hat\phi:A_{C_m}\to A_{C_n}$ by combining $\phi_1$ and $\phi_2$ as follows: if there is a reduced word representing $g$ that starts with $s_0^{\pm 1}$, then set $\hat\phi(g)=\phi_2(g)$; otherwise set $\hat\phi(g)=\phi_1(g)$.

\begin{proposition} \label{prop:no_homo_q0}
The map $\hat\phi:A_{C_m}\to A_{C_n}$ is a quasiisometric embedding that is not at finite distance from any map of the form $\psi_{C_n}\circ h\circ \psi_{C_m}$, where $\psi_{C_n}:A_{C_n}\to A_{C_n}$ and $\psi_{C_m}:A_{C_m}\to A_{C_m}$ are quasiisometric embeddings and $h:A_{C_m}\to A_{C_n}$ is a homomorphism.
\end{proposition}

\begin{proof}
By \cite[Lem.~7.3]{baderbensaidpetyt:quasiisometric:rigidity} and \cite[Prop.~7.11]{baderbensaidpetyt:quasiisometric:rigidity}, for every standard generator $s$ of $A_{C_m}$ there is a standard generator $t_s$ of $A_{C_n}$ such that $\psi_{C_n}\circ h\circ \psi_{C_m}$ maps standard geodesics of type $s$ within finite Hausdorff distance of standard geodesics of type $t_s$. If $\hat \phi$ was at finite distance from such a map, it would satisfy this property as well. But note that $\hat\phi(s_2^n)=t_2t_0^nt_2^{-1}$, whereas $\hat\phi(s_0s_2^n)=t_0t_2^n$.


It remains to show that $\hat\phi$ is a quasiisometric embedding. By \cref{prop:building_cycles_qie} and the fact that $\hat{\sigma}$ and $\hat\tau$ are isometries, both $\phi_1$ and $\phi_2$ are quasiisometric embeddings. Let $(K,A)$ be common quasiisometric constants. We will show that $\hat{\phi}$ is a $(K,2A)$--quasiisometric embedding.

For the upper bound, consider $g\in A_\Gamma$ and $s\in S^{\pm1}$. If $\hat\phi(g)=\phi_i(g)$ and $\hat\phi(gs)=\phi_i(gs)$ for some $i\in \set{1,2}$, then $\dist(\hat\phi(g),\hat\phi(gs))\leq K+A$. 

If $\hat\phi(g)=\phi_2(g)$ and $\hat\phi(gs)=\phi_1(gs)$, then $s$ must cancel with an initial $s_0^a$ syllable of $g$, so we must have $g=s^{-1}g'$, where $s=s_{0}^{\pm1}$ commutes with $g'$ and $g'$ cannot start with $s_0^{\pm 1}$. This implies that only $s_1$ and $s_{-1}$ appear in syllable-reduced expressions for $g'$. 
By definition, $\phi_1(g)=\phi_2(g)$ and $\phi_1(gs)=\phi_2(gs)$, so $\dist(\hat\phi(g),\hat\phi(gs))\leq K+A$.

Similarly, if $\hat\phi(g)=\phi_1(g)$ and $\hat\phi(gs)=\phi_2(gs)$, then we must have $s=s_0^{\pm 1}$ and $g$ is built out of $s_{1}^{\pm 1}$ and $s_{-1}^{\pm 1}$. Again, $\phi_1(g)=\phi_2(g)$ and $\phi_1(gs)=\phi_2(gs)$ and the upper bound was established.
We get that $\hat\phi$ is Lipschitz. 

For the lower bound, let $g_1,g_2\in A_\Gamma$. As before, if $\hat\phi(g_1)=\phi_i(g_1)$ and $\hat\phi(g_2)=\phi_i(g_2)$ for the same $i\in \set{1,2}$, then $\dist(\hat\phi(g_1),\hat\phi(g_2))\geq \frac{1}{K}\dist(g_1,g_2)-A$.
Otherwise, without loss of generality $\hat\phi(g_1)=\phi_1(g_1)$ and $\hat\phi(g_2)=\phi_2(g_2)$. 
This means that $g_2$ can be written as a syllable-reduced word $s_0^ag_2'$ where $a\neq 0$ and $g_1$ cannot start with $s_0^{\pm 1}$. Let $g_1=g_0g_1'$ where $g_0$ is the maximal prefix that commutes with $s_0$. As $g_0$ commutes with $s_0$, we have that $\phi_1(g_0)=\phi_2(g_0)$.

Observe that 
\[
\dist(g_1,g_2) \,=\, |g_2^{-1}g_1| \,=\, |g_2'^{-1}g_0s_0^{-a}g_1'| \,=\, |g_2'^{-1}g_0|+|s_0^{-a}g_1'|.
\]
We will show that a similar thing happens after applying $\hat\phi$ and then use the fact that $\phi_1(g_0)=\phi_2(g_0)$ commutes with $\phi_1(s_0^a)=\phi_2(s_0^a)$ to establish the lower bound.

We have $\hat\phi(g_2)^{-1}\hat\phi(g_1)=\phi_2(g_2)^{-1}\phi_1(g_1)$ by assumption. (Note that it may not be true that $\hat\phi(g_2)^{-1}=\hat\phi(g_2^{-1})$, but we are not claiming this.) Both $\phi_1$ and $\phi_2$ are homomorphisms, so 
\begin{equation*}
    \begin{split}
        \hat\phi(g_2)^{-1}\hat\phi(g_1) \,&=\, \phi_2(g_2'^{-1})\phi_2(s_0^{-a})\phi_1(g_0)\phi_1(g_1') \\
        &=\, \phi_2(g_2'^{-1})\phi_2(g_0)\phi_1(s_0^{-a})\phi_1(g_1') \\ 
        &=\, \phi_2(g_2'^{-1}g_0)\phi_1(s_0^{-a}g_1').
    \end{split}
\end{equation*}
As $g_0$ is a word in $s_1$ and $s_{-1}$, we have that $g_2'^{-1}g_0$ cannot end in $s_0^{\pm 1}$ by the choice of $a$. So $\phi_2(g_2'^{-1}g_0)$ cannot end in $t_0$. On the other hand, $g_1'$ cannot start with $s_{-1}$, $s_0$, or $s_{1}$ as $g_0$ is a maximal prefix that commutes with $s_0$. Hence $\phi_1(s_0^{-a}g_1')$ has a unique first letter $t_0$. This shows that $\phi_2(g_2'^{-1}g_0)\phi_1(s_0^{-a}g_1')$ is reduced, and so 
\begin{equation*}
    \begin{split}
        |\phi_2(g_2'^{-1}g_0)\phi_1(s_0^{-a}g_1')| \,&=\, |\phi_2(g_2'^{-1}g_0)|+ |\phi_1(s_0^{-a}g_1')| \\
        &\geq\, \frac{1}{K}\dist(g_2',g_0)-A+\frac{1}{K}\dist(s_0^{-a}g_1')-A \\
        &=\, \frac{1}{K}\dist(g_1,g_2)-2A,
    \end{split}
\end{equation*}
which establishes the lower bound.
\end{proof}

\section{Cyclic graph products of finite or cyclic groups} \label{sec:cyclic_graph_products}

Here we show that by making small tweaks to the constructions of Section~\ref{sec:building_qie} we can obtain quasiisometric embeddings between more than just right-angled Artin groups. Specifically, we consider graph products of finite groups where the underlying graphs are cycles of length greater than six. As commented on in \cref{rem:n_more_than_six}, we believe that the arguments of Section~\ref{sec:building_qie} can be modified to apply to all cycles of length greater than four, and such modifications should carry through to our considerations here. See also \cref{rem:refinements}.

Let $m,n>6$ be distinct integers such that we can fix an integral expression $m=n+p(n-4)+q(n-2)$, with $p\ge1$ and $q\ge0$, as in Section~\ref{sec:building_qie}. We start by making two observations about the arguments in that section. The first shows that Proposition~\ref{prop:roots_structure} and Lemma~\ref{lem:prefixes} simplify in the case $q=0$.

\begin{lemma} \label{obs:no_glides}
Suppose that $q=0$. If $\gamma$ is a lazy word representing an element $g\in A_{C_m}$, then $F(\gamma)$ is syllable-reduced. Moreover, if $h\in A_{C_m}$ is an $A_{C_m}$-prefix of $g$, then $F(\gamma)$ can be shuffled to obtain a representative word $\mu$ of $F(g)$ with an initial subword that minimally represents $F(h)$. 
\end{lemma}

\begin{proof}
If $q=0$, then there are no shortcuts or glides involved in the constructions of the maps $f$ and $F$ in Section~\ref{sec:building_qie}. The arguments of that section therefore simplify: we do not need Lemmas~\ref{lem:stay_shortcut} or~\ref{lem:gliding_commutes}, for instance. 

The statement that $F(\gamma)$ is syllable-reduced can be seen by following the proof of Proposition~\ref{prop:roots_structure} and using the fact that there are no glides: as noted in the second paragraph thereof, adding a new final syllable can only cause a failure to be syllable-reduced if there is a glide. See Cases~\hyperlink{structure_case_1}{1} and~\hyperlink{structure_case_2c}{2.c}.

The statement about prefixes of $g$ can thus be seen by following the start of the proof of Lemma~\ref{lem:prefixes} and using the fact that $F(\gamma)$ is syllable-reduced. Indeed, if $F(\gamma)$ is syllable-reduced then its length is equal to $|F(g)|$, so, as stated in the fourth paragraph of the proof of that lemma, some shuffle of $F(\gamma)$ has an initial subword minimally representing $F(h)$.
\end{proof}

\begin{observation} \label{obs:labels_powers}
Regardless of the value of $q$, Proposition~\ref{prop:roots_structure} characterises the ways in which $F(\gamma)$ can fail to be syllable-reduced, where $\gamma$ is a lazy representative of an element $g\in A_{C_m}$ and $F\colon A_{C_m}\to A_{C_n}$ is the map constructed in Item~\ref{sh:roots}. In particular, this requires there to be a shortcut ``on the way to $g$''. 

When $F(\gamma)$ does fail to be syllable-reduced, the failure is caused by a pair of basic relative labels $\ell_c$ and $\ell_{c+k+1}$ that appear in the expression for $F(\gamma)$, as described in Proposition~\ref{prop:roots_structure}. If we were to change the powers of the letters appearing in the basic relative labels, then the only change to the description of failures of being syllable-reduced in Proposition~\ref{prop:roots_structure} would be in the description of the powers that appear in $\ell_c$ and $\ell_{c+k+1}$. The construction of $\lambda'$ would then involve shuffling the final syllable of $\ell_c$, rather than just the final letter (the two are the same in the setting of Proposition~\ref{prop:roots_structure}), and the exponents may not cancel.
\end{observation}

Let $L=\{2,3,4,\dots,\infty\}$. For a graph $\Gamma$, let $L^\Gamma$ denote the set of vectors $\bar N=(N_v)_{v\in\Gamma}$ with $N_v\in L$ for all $v$. Given $\bar N\in L^\Gamma$, let $N^+=\max\{N_v\,:\,v\in\Gamma\}\in L$, and let $N^-=\min\{N_v\,:\,v\in\Gamma\}\in L$. If $N^-<\infty$ then let $N'=\max\{N_v\,:\,N_v<\infty\}$. 

We define $A_\Gamma(\bar N)$ to be the congruence quotient of $A_\Gamma$ obtained by declaring each standard generator $s_v$ to have order $N_v$. In other words, $A_\Gamma(\bar N)$ is the graph product defined by putting $\sfrac{\Z}{N_v\Z}$ on each vertex of the graph $\Gamma$, where we here regard $\sfrac\Z{\infty\Z}=\Z$. 

The standard generating set of $A_\Gamma$ induces a set $S$ of generators of $A_\Gamma(\bar N)$. For $g\in A_\Gamma(\bar N)$, we write $|g|$ for the word-length of $g$ with respect to these generators. We shall say that a word $\gamma$ representing $g\in A_\Gamma(\bar N)$ is \emph{syllable-reduced} if it has the form $s_1^{n_1}\dots s_k^{n_k}$, with $k$ the syllable-length of the element of $A_\Gamma$ represented by $\gamma$, and where $n_i\in\{1,\dots,N_v-1\}$ whenever $s_i=s_v$ has $N_v<\infty$.

We also have natural semigroup generators $\{s_v^{\pm1}\,:\,N_v=\infty\}\cup\{s_v\,:\,N_v<\infty\}$. If $\|g\|$ denotes the word-length of $g$ with respect to these semigroup generators, then
\[
|g| \,\le\, \|g\| \,<\, N'|g|.
\]
Every $g\in A_\Gamma(\bar N)$ has a unique lift $\hat g\in A_\Gamma$ with the property that $\hat g$ has a syllable-reduced expression whose $s_v$-exponents all lie in $\{1,\dots,N_v-1\}$ if $N_v<\infty$. From the above expression, we have $|g|\le|\hat g|\le N'|g|$.

Let $\Gamma^\ext(\bar N)$ denote the quotient of $\Gamma^\ext$ by the action of $\ker(A_\Gamma\to A_\Gamma(\bar N))$. The graph $\Gamma^\ext(\bar N)$ has an obvious section to $\Gamma^\ext$, consisting of all vertices $g\cdot v$ such that $g\in A_\Gamma$ can be written as a syllable-reduced word $s_1^{a_1}\dots s_d^{a_d}$ such that $a_i\in\{1,\dots,N_v-1\}$ for all $i$ such that $s_i=s_v$ with $N_v<\infty$. By an abuse of notation, we identify $\Gamma^\ext(\bar N)$ with its image under this section and view $A_\Gamma(\bar N)$ as acting on its image in $\Gamma^\ext$. 



We can now prove our main result on graph products of cyclic groups. We will deduce a similar statement for graph products of finite groups in \cref{cor:graph_prod_finite}.

\begin{theorem} \label{thm:cyclic_product_cyclic}
Suppose that $m,n>6$ are integers with $m=n+p(n-4)+q(n-2)$ for some integers $p\ge1$ and $q\ge0$. Let $\bar M\in L^{C_m}$, $\bar N\in L^{C_n}$. If any of the following hold, then there is a quasiisometric embedding $A_{C_m}(\bar M)\to A_{C_n}(\bar N)$.
\begin{itemize}
\item   $N^-=\infty$.
\item   $M^+<N^-$.
\item   $M^+=N^-<\infty$ and $q=0$.
\end{itemize}
Moreover, if $q=0$ and $N^-=N^+=kM^-=kM^+$ for some positive integer $k$, then $A_{C_m}(\bar M)$ is an undistorted subgroup of $A_{C_n}(\bar N)$.
\end{theorem}

\begin{proof}
Let us write $\Gamma=C_m$ and $\Lambda=C_n$. Let $\{s_v\,:\,v\in\Gamma\}$ and $\{t_w\,:\,w\in\Lambda\}$ be the standard generating sets of $A_\Gamma(\bar M)$ and $A_\Lambda(\bar N)$, respectively.

Consider the map $f\colon\Gamma^\ext\to\Lambda^\ext$ constructed in Item~\ref{sh:f_construction}. Supposing that $N^-=\infty$, recall the map $F\colon A_\Gamma\to A_\Lambda$ constructed using $f$ in Item~\ref{sh:roots}. Let $F'$ be the restriction of $F$ to (the section in $A_\Gamma$ of) $A_\Gamma(\bar M)$. If instead $N^-<\infty$, then we shall first modify $f$, and then use the modified map to construct a map $F':A_\Gamma(\bar M)\to A_\Lambda(\bar N)$.

Suppose that $N^-<\infty$. By our assumptions, we therefore have $M^+<\infty$. We modify $f$ by replacing every instance of $t_w^{-1}$ in a basic relative label by $t_w^{N_w-1}$, for every standard generator $t_w$ of $A_\Lambda$ for which $N_w<\infty$. This gives us a different map $f':\Gamma^\ext\to\Lambda^\ext$. Note that composing either $f$ or $f'$ with the quotient $\Lambda^\ext\to\Lambda^\ext(\bar N)$ gives the same map. 

If $M^+<N^-$, then we can see from the construction in Item~\ref{sh:f_construction} that $f'$ restricts to a map $f_{\bar M}:\Gamma^\ext(\bar M)\to\Lambda^\ext(\bar N)$. Additionally, Lemma~\ref{obs:no_glides} shows that the same is true if $M^+=N^-$ and $q=0$.

Starting from the map $f'$, we recursively define a map $F':A_\Gamma\to A_\Lambda$ in the same way as in Item~\ref{sh:roots}, with $f'$ replacing $f$. That is, for a reduced word $\gamma's^d$, we define $F(\gamma's^d)=F(\gamma')\ell t^{d}$ where $\ell$ is the minimal basic relative label and $t$ is a generator, both determined by $f'$. The same argument as in Lemma~\ref{lem:F_well_defined} shows that $F'$ is well defined. According to Observation~\ref{obs:labels_powers}, Proposition~\ref{prop:roots_structure} holds for $F'$, with the difference that the exponents appearing in the relative labels $\ell_i$ are either 1 or $N_w-1$ whenever $N_w<\infty$, rather than always being $\pm1$, which they still are when $N_w=\infty$.   

Suppose that $g\in A_\Gamma$ has a syllable-reduced expression whose $s_v$-exponents all lie in $\{1,\dots,M_v-1\}$ for each $v\in\Gamma$, and let $\gamma$ be a lazy representative of $g$. As a consequence of the above modified version of Proposition~\ref{prop:roots_structure}, we see that all the $t_w$-exponents appearing in $F'(\gamma)$ lie in the set $\{1,\dots,M^+-1,N_w-1\}$. Note that we are not claiming here that $F'(\gamma)$ is syllable-reduced.

Regardless of whether $N^-<\infty$, we can now define a map $F_{\bar M,\bar N}:A_\Gamma(\bar M)\to A_\Lambda(\bar N)$ as follows. Given $g\in A_\Gamma(\bar M)$, consider the unique lift $\hat g\in A_\Gamma$ of $g$ whose $s_v$-exponents all lie in $\{1,\dots,M_v-1\}$ whenever $M_v<\infty$. Let $\gamma$ be a lazy representative of $\hat g$. We let $F_{\bar M,\bar N}(g)$ be the image of $F'(\gamma)$ in $A_\Lambda(\bar N)$ under the quotient $A_\Lambda\to A_\Lambda(\bar N)$.

\medskip

Let us prove that $F_{\bar M,\bar N}$ is a quasiisometric embedding. Let $\gamma$ be a syllable-reduced word representing $g\in A_\Gamma(\bar M)$. If $M^-=\infty$, then $A_\Gamma(\bar M)=A_\Gamma$, and the length of $\gamma$ is equal to $|g|$. If $M^-<\infty$, then the length of $\gamma$ is at least $|g|$ and at most $(M'-1)|g|$, where we recall that $M'=\max\{M_v\,:\,M_v<\infty\}$. Thus, even if $M^-<\infty$, the lifting map $g\mapsto\hat g$ is a quasiisometric embedding, with constant at most $M'$. 

If $N^-<\infty$, then, as noted above, the modified version of Proposition~\ref{prop:roots_structure} shows that if $\gamma$ is a lazy representative of $\hat g$ for some $g\in A_\Gamma(\bar M)$, then $F'(\gamma)$ has its $t_w$-exponents lying in $\{1,\dots,M'-1,N_w-1\}$ whenever $N_w<\infty$. Thus the restriction of the quotient map $A_\Lambda\to A_\Lambda(\bar N)$ to the $F'$-image of the lift of $A_\Gamma(\bar M)$ is a quasiisometric embedding, where the constant is 1 if $N^-=\infty$, and at most $N'$ if $N^-<\infty$.

As a consequence of these considerations, it suffices to show that $F':A_\Gamma\to A_\Lambda$ is a quasiisometric embedding. For $N^-=\infty$, this is precisely Proposition~\ref{prop:building_cycles_qie}.

For $N^-<\infty$, the argument that $F'$ is a quasiisometric embedding is the same as in Proposition~\ref{prop:building_cycles_qie} with only superficial differences due to our tweak in the definition of $f'$, which we now highlight. 
Firstly, the maximal length of a basic relative label is bounded by $(N'-1)(p+2q)$, which affects the Lipschitz constant of $F'$ and the additive coarse colipschitz constant. 

The second difference actually marginally simplifies the proof of the colipschitz property. Let $\gamma$ be a lazy representative of some $g\in A_\Gamma$, and consider the possible ways for $F(\gamma)$ to fail to be syllable-reduced, as described in \cref{prop:roots_structure}. Recall that, in replacing $F$ by $F'$, all instances of $t_w^{-1}$ appearing in basic relative labels have been replaced by $t_w^{N_w-1}$ when $N_w<\infty$. It follows that such $t_w$ cannot cause a failure for $F(\gamma)$ to be syllable-reduced in the way described in \cref{prop:roots_structure}. Thus if $h$ is an $A_\Gamma$-prefix of $g$, then we can follow the proof of Lemma~\ref{lem:prefixes} and there will be fewer cancelling pairs that need shuffling to produce the representative $\mu$ of $F'(g)$ with an initial subword $\nu$ representing $F'(h)$.

With these changes accounted for, the proof of \cref{prop:building_cycles_qie} follows through to show that $F'$ is a quasiisometric embedding. As described above, this shows that $F_{\bar M,\bar N}:A_\Gamma(\bar M)\to A_\Lambda(\bar N)$ is a quasiisometric embedding.

\medskip

We now consider the second statement, in which $q=0$. Let $M=M^-=M^+$, and let $N=N^-=N^+=kM$. We make a small modification to the above construction. Define a modified map $f_{M,k}:\Gamma^\ext(\bar M)\to\Lambda^\ext(\bar N)$ by following Item~\ref{sh:f_construction}, except, in the notation of that item, first set $f|_{g\cdot \Gamma}=D^{ka}_{f(g'\cdot v_i)}(B)$ instead of $D^a_{f(g'\cdot v_i)}(B)$, and then change the powers appearing in basic relative labels as in the construction of $f'$ above. Similarly to Lemma~\ref{obs:no_glides}, the lemmas subsequent to Item~\ref{sh:f_construction} go through to show that $f_{M,k}$ is well defined, and as above we can associate a map $F_{M,k}:A_\Gamma(\bar M)\to A_\Lambda(\bar N)$. The same argument as above shows that $F_{M,k}$ is a quasiisometric embedding. 

To show that $F_{M,k}$ is at finite distance from a homomorphism, we define a perturbation $\phi_{M,k}$ of $F_{M,k}$ in the same way as in Lemma~\ref{lem:homo}. To show that $\phi_{M,k}$ is a homomorphism, we use the same argument as in that lemma, plus the following two observations: firstly, if $a+b=M$, then $\phi(s^a)\phi(s^b)=1$ for every standard generator $s\in S$; and secondly, wherever we see an $\ell^{-1}$ in the expressions of Lemma~\ref{lem:homo}, that expression is accurate as a representative of an element of $A_\Lambda(\bar N)$ because of the way we tweaked the powers in basic relative labels.
\end{proof}

\begin{corollary} \label{cor:graph_prod_finite}
Suppose that $m,n>6$ are integers with $m=n+p(n-4)+q(n-2)$ for some integers $p\ge1$ and $q\ge0$. Let $M,N\in\N_{\ge2}$. Let $G$ and $H$ be graph products of finite or cyclic groups whose underlying graphs are $C_m$ and $C_n$, respectively, and whose vertex groups all have order at most $M$ and at least $N$, respectively. 

If $M<N$, or if $M=N$ and $q=0$, then there is a quasiisometric embedding $G\to H$.
\end{corollary}

\begin{proof}
For $v\in C_m$, let $M_v$ be the order of the vertex group at $v$, and let $\bar M=(M_v)_{v\in C_m}$. Define $\bar N=(N_w)_{w\in C_n}$ similarly. According to \cite[Thm~8.22]{genevois:cubical}, the group $G$ is quasiisometric to $A_{C_m}(\bar M)$, and $H$ is quasiisometric to $A_{C_n}(\bar N)$. The statement therefore follows from \cref{thm:cyclic_product_cyclic}.
\end{proof}


\begin{remark} \label{rem:refinements}
The assumptions of Theorem~\ref{thm:cyclic_product_cyclic} can likely be refined if one takes into account more information about the map $f$ in Item~\ref{sh:f_construction}. Such refinements of \cref{thm:cyclic_product_cyclic} would yield corresponding refinements of \cref{cor:graph_prod_finite}. Certainly one can, with more care, give conditions to obtain group embeddings in more generality. 
\end{remark}

For a constant vector $\bar N=(N)_{v\in\Gamma}\in L^\Gamma$, let us write $A_\Gamma(N)=A_\Gamma(\bar N)$. 
If $n>4$, then the hyperbolic plane can be tiled by regular right-angled $n$-gons. A regular such tiling can be generated by the reflection group $A_{C_n}(2)$. More generally, $A_{C_n}(\bar N)$ is a uniform lattice in the Bourdon building $I_{n,\bar N}$.

In contrast to the rigidity of quasiisometries between Bourdon buildings discussed in the introduction, Theorem~\ref{thm:cyclic_product_cyclic} yields exotic quasiisometric embeddings. We conclude by describing exotic quasiisometric embeddings of $\mathbb{H}^2$ into Bourdon buildings and right-angled Artin groups.

\begin{corollary}
If $\Gamma$ contains an induced cycle of length greater than four, then there are quasiisometric embeddings of $\mathbb{H}^2$ into $A_\Gamma$ with image not contained in any finite neighbourhood of any finite union of surface subgroups.
\end{corollary}

\begin{proof}
Suppose that $\Gamma$ contains an induced cycle $C$ of length greater than four. As described at the start of Section~\ref{sec:induced_implies_subgroup}, this implies that $A_\Gamma$ has a finite-index subgroup with a parabolic subgroup of the form $A_{C_n}$ for some $n\ge10$. Indeed, $C_{|C|+5(|C|-4)}$ can easily be found as an induced subgraph of $C^\ext\subset\Gamma^\ext$. 

Let $m=n+(n-2)+(n-4)$. The group $A_{C_m}(2)$ is a virtual surface group, so we can let $F_{2,\infty}:A_{C_m}(2)\to A_{C_n}$ be the quasiisometric embedding constructed in the proof of Theorem~\ref{thm:cyclic_product_cyclic}. More precisely (since $A_{C_n}=A_{C_n}(\infty)$), the map $F_{2,\infty}$ is defined by lifting $g\in A_{C_m}(2)$ to $\hat g\in A_{C_m}$ and then applying the map $F$ constructed in Item~\ref{sh:roots}. We shall show that $F_{2,\infty}$ is the desired quasiisometric embedding. The phenomenon is essentially the same as in Lemma~\ref{lem:no_homo}, but more care is needed because the generators of $A_{C_m}(2)$ are involutions.

Using the standard generating set $S$ with the cyclic convention as described at the start of Section~\ref{sec:building_qie}, consider the elements $g=s_3s_{-3}$ and $h=s_1s_{-1}$ in $A_{C_m}(2)$. If $G<A_{C_m}(2)$ has index $r$, then $g^{r!k}\in G$ and $h^{r!k}\in G$ for all $k\in\Z$.

From the description in Item~\ref{sh:roots}, we can compute 
\[
F(\hat g^k) \,=\, t_2t_{-1}^2t_5(t_{-1}t_5)^{k-1}.
\]
Indeed, $f(v_3)$ is a shortcut of the block $f(C_m)$, so $F(s_3)=t_2t_{-1}^2$, and then we have that $f(s_3\cdot v_{-3})$ and $f(s_3\cdot v_3)$ both lie in $F(s_3)\cdot C_n$, so there are no further glides when computing $F(\hat g^k)$. Thus $F_{2,\infty}(g^k)=t_2t_{-1}(t_{-1}t_5)^k$. More easily seen is that $F_{2,\infty}(h^k)=(t_1t_{-1})^k$. Hence, if $\psi:A_{C_m}(2)\to A_{C_n}$ is a map at distance at most $p\in\mathbb N$ from $F_{2,\infty}$, then there are elements $z,w\in A_{C_n}$ with $|z|,|w|\le p$ such that
\[
\psi(g^p)\psi(h^p) \,=\, t_2t_{-1}(t_{-1}t_5)^pz(t_1t_{-1})^pw.
\]

On the other hand, by building on the computation of $F(\hat g^k)$ given above, we find that 
\[
F_{2,\infty}(g^ph^p) \,=\, t_2t_{-1}(t_{-1}t_5)^p(t_1t_3)^p.
\]
It is easy to see that there is no element of $A_{C_n}$ that lies within a distance of $p$ of both $F_{2,\infty}(g^ph^p)$ and $\psi(g^p)\psi(h^p)$, regardless of which elements $z$ and $w$ of length at most $p$ are chosen.

By taking $p$ to be increasingly large multiples of $r!$, we conclude that $F_{2,\infty}$ is not at finite distance from a homomorphism on any finite-index subgroup of $A_{C_m}(2)$. Since all surface groups are commensurable, this shows that $F_{2,\infty}(A_{C_m}(2))$ does not lie at finite distance from any surface subgroup of $A_\Gamma$. 

The elements $g$ and $h$ above were chosen for their simplicity, but one can easily find infinitely many such pairs. By varying the pairs one obtains the stronger statement that $F_{2,\infty}(A_{C_m}(2))$ is not contained in a finite neighbourhood of a union of finitely many surface subgroups of $A_\Gamma$.
\end{proof}

An analogous proof shows the following.

\begin{corollary}
Let $n>6$. If $\bar N=(N_1,\dots,N_n)$ has $N_i\in\{3,4,\dots\}$ for all $i$, then the Bourdon building $I_{n,\bar N}$ admits quasiisometric embeddings of $\mathbb H^2$ that do not lie at finite distance from surface subgroups and whose intersection with each standard copy of $\mathbb H^2$ is bounded.
\end{corollary}

\bibliographystyle{alpha}
\footnotesize{\bibliography{biblio}}
\Addresses
\end{document}

%% file: preamble.tex
\usepackage[leqno]{amsmath}
\usepackage{amssymb,xfrac}
\usepackage{amsthm}
\usepackage{ifthen} 
\usepackage[pagebackref,breaklinks,unicode]{hyperref} 
    \renewcommand*{\backrefalt}[4]{\ifcase #1 (Not cited).\or (Cited p.~#2).\else (Cited pp.~#2).\fi} 
\usepackage[capitalise,noabbrev]{cleveref}
\usepackage{enumitem}
\usepackage{graphicx}
\usepackage{tikz-cd, tikz}
\usepackage{float}
\usepackage[justification=centering]{caption}
\usepackage{subcaption}
\usepackage{wasysym}

\numberwithin{equation}{section}

\newtheorem{theorem}{Theorem}[section]
\newtheorem{proposition}[theorem]{Proposition}
\newtheorem{lemma}[theorem]{Lemma}
\newtheorem{corollary}[theorem]{Corollary}

\newtheorem*{theorem*}{Theorem}
\newtheorem*{proposition*}{Proposition}
\newtheorem*{lemma*}{Lemma}
\newtheorem*{corollary*}{Corollary}

\newtheorem{mthm}{Theorem} 
\newtheorem{mcor}[mthm]{Corollary}

\theoremstyle{definition}
\newtheorem{definition}[theorem]{Definition}
\newtheorem{observation}[theorem]{Observation}
\newtheorem{remark}[theorem]{Remark}

\newtheorem*{definition*}{Definition}
\newtheorem*{observation*}{Observation}
\newtheorem*{remark*}{Remark}
\newtheorem*{example*}{Example}
\newtheorem*{question*}{Question}
\newtheorem*{exercise*}{Exercise}
\newtheorem*{fact*}{Fact}
\newtheorem*{notation*}{Notation}

\newcommand{\Z}{\mathbb{Z}}

\DeclareMathOperator{\str}{Star}

\DeclareMathOperator{\starr}{Star}
\renewcommand*{\star}{\starr}

\newcommand{\set}[1]{\left\{ #1 \right\}}

\newcounter{claimref}
\newcounter{claimcount}

\newenvironment{claim*}{\par\medskip\noindent\textbf{Claim:}\hspace{0.5mm}}{}

\newenvironment{claim*proof}{\medskip\noindent\emph{Proof of Claim.}\hspace{0.5mm}}
    {\leavevmode\unskip\penalty9999\hbox{}\nobreak\hfill\quad\hbox{$\diamondsuit$}\medskip}

\newcounter{subclaimcount}

\newenvironment{subclaim*}{\par\medskip\textbf{Subclaim:}\hspace{0.5mm}}{}
\newenvironment{subclaimproof*}{\medskip\noindent\emph{Proof of Subclaim.}\hspace{0.5mm}}
    {\leavevmode\unskip\penalty9999\hbox{}\nobreak\hfill\quad\hbox{$\fullmoon$}\medskip}

\crefname{cond}{condition}{conditions}
\creflabelformat{cond}{#2#1\@#3}
\crefname{obs}{observation}{observations}
\creflabelformat{obs}{#2#1\@#3}

\renewcommand{\cal}{\mathcal}
\newcommand*{\eps}{\varepsilon}
\renewcommand*{\L}{\mathcal{L}}

\newcommand*{\dist}{d}

\DeclareMathOperator{\dep}{Dep}

\DeclareMathOperator{\doub}{D}
\DeclareMathOperator{\glide}{Gl}

\DeclareMathOperator{\short}{Short}
\DeclareMathOperator{\type}{Typ}

\newcommand*{\syl}[1]{|#1|_{\mathrm{syl}}}

\renewcommand*{\subset}{\subseteq}

\newcounter{shcount}
\newcounter{thmcount}

\swapnumbers
\newcommand*{\bsh}[1]{\theoremstyle{definition}\newtheorem{subhead\theshcount}[theorem]{#1}
    \begin{subhead\theshcount}} 
\newcommand*{\esh}{\end{subhead\theshcount}\stepcounter{shcount}} 
 
\newcommand*{\numberedtheorem}[3]{\theoremstyle{plain}\newtheorem*{makethm\thethmcount}{#1}
    \ifthenelse{\equal{#2}{}}{\begin{makethm\thethmcount}#3\end{makethm\thethmcount}\stepcounter{thmcount}}
    {\begin{makethm\thethmcount}[#2]#3\end{makethm\thethmcount}\stepcounter{thmcount}}} 

\makeatletter\newcommand{\linkdest}[1]{\Hy@raisedlink{\hypertarget{#1}{}}}\makeatother 

\usepackage{lmodern,anyfontsize} 
\usepackage{setspace} 
\usepackage{tocloft} 
\usepackage[nottoc]{tocbibind}